\documentclass[10pt, a4paper,final]{amsart}

\usepackage{amsmath,amsfonts,mathscinet,eucal,amsthm,amssymb,bm,extarrows,upgreek,tensor,mathrsfs,textcomp,comment,mathtools,braket,wasysym}

\mathtoolsset{mathic}

\usepackage{array}

\usepackage[shortlabels,inline]{enumitem}

\usepackage[final]{microtype}

\usepackage{ifpdf}
\ifpdf
\usepackage[final=true]{hyperref,bookmark}
\fi
\usepackage{cleveref}
\crefformat{subsection}{#2\textsection#1#3}
\Crefformat{subsection}{#2Section#1#3}
\crefmultiformat{subsection}{#2\textsection#1#3}{ and #2\textsection#1#3}{, #2\textsection#1#3}{ and #2\textsection#1#3}
\Crefmultiformat{subsection}{#2Sections#1#3}{ and~#2#1#3}{, #2#1#3}{ and~#2#1#3}
\crefrangeformat{enumi}{#3#1#4--#5#2#6}
\crefrangeformat{equation}{(#3#1#4--#5#2#6)}
\crefrangeformat{subequation}{(#3#1#4--#5\crefstripprefix{#1}{#2}#6)}
\crefname{enumi}{}{}
\Crefname{enumi}{}{}
\crefname{equation}{}{}

\numberwithin{equation}{section}
\usepackage{fontenc} \usepackage[utf8]{inputenc}

\usepackage{comment}

\newtheoremstyle{myplain} {6pt plus 6pt minus 2pt}
{6pt plus 6pt minus 2pt}
{\itshape}
{}
{\bfseries}
{.}
{.5em}
{}

\theoremstyle{myplain}
\newtheorem{theorem}{Theorem}[section]
\newtheorem*{theorem*}{Theorem}
\newtheorem{lemma}[theorem]{Lemma}
\newtheorem{prop}[theorem]{Proposition}
\newtheorem{corollary}[theorem]{Corollary}

\newtheoremstyle{mydefinition} {6pt plus 6pt minus 2pt}
{6pt plus 6pt minus 2pt}
{\itshape}
{}
{\bfseries}
{.}
{.5em}
{}

\theoremstyle{mydefinition}
\newtheorem{definition}[theorem]{Definition}

\newtheoremstyle{myexample} {6pt plus 6pt minus 2pt}
{6pt plus 6pt minus 2pt}
{}
{}
{\scshape}
{.}
{.5em}
{}

\theoremstyle{myexample}
\newtheorem{example}[theorem]{Example}

\newtheoremstyle{myremark} {6pt plus 6pt minus 2pt}
{6pt plus 6pt minus 2pt}
{}
{}
{\scshape}
{.}
{.5em}
{}

\theoremstyle{myremark}
\newtheorem{remark}[theorem]{Remark}

\newcommand{\Z}{\mathbb{Z}}

\newcommand{\C}{\mathbb{C}}

\newcommand{\gl}{\mathfrak{gl}}
\newcommand{\catO}{\mathcal{O}}

\newcommand{\suchthat}{\mid} 
\newcommand{\mapto}{\rightarrow}
\newcommand{\surto}{\twoheadrightarrow}
\newcommand{\into}{\hookrightarrow}
\newcommand{\len}{\ell}

\newcommand{\blank}{\mathord{-}}

\DeclareMathOperator{\Hom}{Hom}

\DeclareMathOperator{\End}{End}

\DeclareMathOperator{\coker}{coker}

\newcommand{\id}{\mathrm{id}}

\newcommand{\abs}[1]{\left|#1\right|}

\newcommand{\gmod}[1]{#1\mathrm{-gmod}}

\newcommand{\rgmod}[1]{\mathrm{gmod-}#1}
\newcommand{\lmod}[1]{#1\mathrm{-mod}}
\newcommand{\rmod}[1]{\mathrm{mod-}#1}

\renewcommand{\epsilon}{\varepsilon}

\renewcommand{\phi}{\varphi}

\newcommand{\calA}{{\mathcal{A}}}
\newcommand{\calB}{{\mathcal{B}}}
\newcommand{\calC}{{\mathcal{A}}}
\newcommand{\ocalC}{{\mathcal{C}}}

\newcommand{\calF}{{\mathcal{F}}}
\newcommand{\calG}{{\mathcal{G}}}

\newcommand{\calP}{{\mathcal{P}}}
\newcommand{\calQ}{{\mathcal{Q}}}

\newcommand{\calS}{{\mathcal{S}}}

\newcommand{\sfC}{{\mathsf{C}}}

\newcommand{\sfE}{{\mathsf{E}}}
\newcommand{\sfF}{{\mathsf{F}}}
\newcommand{\sfG}{{\mathsf{G}}}

\newcommand{\sfK}{{\mathsf{K}}}

\newcommand{\sfT}{{\mathsf{T}}}

\newcommand{\sfZ}{{\mathsf{Z}}}

\newcommand{\sfi}{{\mathsf{i}}}

\newcommand{\frakb}{{\mathfrak{b}}}

\newcommand{\frakd}{{\mathfrak{d}}}

\newcommand{\frakg}{{\mathfrak{g}}}
\newcommand{\frakh}{{\mathfrak{h}}}

\newcommand{\frakl}{{\mathfrak{l}}}

\newcommand{\frakn}{{\mathfrak{n}}}

\newcommand{\frakp}{{\mathfrak{p}}}
\newcommand{\frakq}{{\mathfrak{q}}}

\newcommand{\fraku}{{\mathfrak{u}}}

\newcommand{\frakz}{{\mathfrak{z}}}

\newcommand{\bbS}{{\mathbb{S}}}

\newcommand{\bbV}{{\mathbb{V}}}

\newcommand{\bbZ}{{\mathbb{Z}}}

\newcommand{\bolda}{{\boldsymbol{a}}}
\newcommand{\boldb}{{\boldsymbol{b}}}
\newcommand{\boldc}{{\boldsymbol{c}}}
\newcommand{\boldd}{{\boldsymbol{d}}}

\newcommand{\boldr}{{\boldsymbol{r}}}

\usepackage[bbgreekl]{mathbbol}

\DeclareSymbolFontAlphabet{\mathbb}{AMSb}
\DeclareSymbolFontAlphabet{\mathbbol}{bbold}

\DeclareMathAlphabet{\mathbbm}{U}{bbm}{m}{n}
\SetMathAlphabet\mathbbm{bold}{U}{bbm}{bx}{n}

\DeclareSymbolFont{usualmathcal}{OMS}{cmsy}{m}{n}
\DeclareSymbolFontAlphabet{\mathucal}{usualmathcal}

\RequirePackage{parskip}
\usepackage{etoolbox}
\preto\subequations{\ifhmode\unskip\fi}

\begingroup
    \makeatletter
\@for\theoremstyle:=mydefinition,myremark,myplain,myexample\do{%
        \expandafter\g@addto@macro\csname th@\theoremstyle\endcsname{%
            \addtolength\thm@preskip\parskip
            }%
        }
\endgroup

\usepackage[usenames,dvipsnames]{xcolor} \usepackage{tikz}
\usetikzlibrary{patterns,matrix,through,arrows,decorations.pathreplacing,decorations.markings,shadows,shapes.geometric,positioning,calc,backgrounds}

\tikzset{anchorbase/.style={baseline={([yshift=-0.5ex]current bounding box.center)}}}
\tikzset{anchorzero/.style={baseline={([yshift=-0.5ex]0,0)}}}

\usepackage{todonotes} 

\usepackage{ytableau}

\newcommand{\short}{{\textrm{short}}}

\newcommand{\shortestcosetleft}[2]{\scalebox{0.9}{$\big(\text{\raisebox{-2pt}{$#1$}}\hspace{-2pt} \rotatebox{10}{$\big\backslash$} \hspace{-2pt}\text{\raisebox{2pt}{$#2$}}\big)^{\short}$}}

\newcommand{\leftquotient}[2]{\scalebox{0.9}{\text{\raisebox{-2pt}{$#1$}}\hspace{-5pt} \rotatebox{10}{$\big\backslash$} \hspace{-5pt}\text{\raisebox{2pt}{$#2$}}}}

\newcommand{\rightquotient}[2]{\scalebox{0.9}{\text{\raisebox{2pt}{\(#1\)}}\hspace{-6pt} \rotatebox{-10}{$\big/$} \hspace{-4pt}\text{\raisebox{-2pt}{\(#2\)}}}}

\newcommand{\adjunction}{\dashv}

\newcommand{\sfEZ}{{\prescript{\smash \Z}{}{\sfE}}}
\newcommand{\sfFZ}{{\prescript{\smash \Z}{}{\sfF}}}
\newcommand{\sfTZ}{{\prescript{\smash \Z}{}{\sfT}}}
\newcommand{\sfKZ}{{\prescript{\smash \Z}{}{\sfK}}}
\newcommand{\DeltaZ}{{\prescript{\smash \Z}{}{\Delta}}}

\newcommand{\sfEg}{{\boldsymbol\sfE}}
\newcommand{\sfFg}{{\boldsymbol \sfF}}
\newcommand{\sfKg}{{\boldsymbol\sfK}}
\newcommand{\xig}{{\boldsymbol \xi}}
\newcommand{\taug}{{\boldsymbol \tau}}
\newcommand{\calCg}{{\boldsymbol \calC}}
\newcommand{\Deltag}{{\boldsymbol \Delta}}

\newcommand{\tableau}{{\mathrm{T}}}

\newcommand{\prel}{\prec}
\newcommand{\preleq}{\preccurlyeq}
\newcommand{\preg}{\succ}

\newcommand{\oDelta}{{\overline{\Delta}}}
\newcommand{\enne}{{n}}

\newcommand{\rhogl}{{\boldsymbol{\rho}}}

\newcommand{\catOZ}{{\prescript{\smash \Z}{}{\catO}}}
\newcommand{\calQZ}{{\prescript{\smash \Z}{}{\calQ}}}

\newcommand{\fraksl}{{\mathfrak{sl}}}

\newcommand{\ucalH}{{\mathucal{H}}}
\newcommand{\ucalC}{{\mathucal{C}}}

\newcommand{\Part}{{\mathfrak{Part}}}  
  
\newcommand{\weightsgl}{{\mathrm{P}}}
\newcommand{\weightssl}{{\mathbbol{\Lambda}}}
 
\newcommand{\boldlambda}{{\boldsymbol{\lambda}}}
\newcommand{\boldnu}{{\boldsymbol{\nu}}}
\DeclareMathOperator{\Vect}{Vect}
\DeclareMathOperator{\pr}{pr}
\DeclareMathOperator{\res}{res}
\DeclareMathOperator{\gr}{gr}
\newcommand{\St}{{\mathrm{St}}}
\newcommand{\boldT}{{\boldsymbol{T}}}
\newcommand{\aff}{{\mathrm{aff}}}

\renewcommand{\blank}{{\mathord{\bullet}}}

\renewcommand{\bolda}{{\mathbf{a}}}
\renewcommand{\boldb}{{\mathbf{b}}}
\renewcommand{\boldc}{{\mathbf{c}}}
\renewcommand{\boldd}{{\mathbf{d}}}

\renewcommand{\boldr}{{\mathbf{r}}}

\DeclareMathOperator{\sgn}{sgn}

\newcommand{\Uqslk}{U_q(\mathfrak{sl}_k)}

\newcommand{\pres}{{\mathrm{pres}}}

\hypersetup{colorlinks = false,
  urlcolor = magenta,
  pdfauthor = {Antonio Sartori and Catharina Stroppel},
  pdfkeywords = {Categorification, Category O, Standardly stratified categories, quantum groups, Representations of sl(k)},
  pdftitle = {Categorification of tensor product representations of sl(k) and category O},
  pdfsubject = {Representation theory},
  pdfpagemode = UseNone,
  bookmarksopen = true,
  bookmarksopenlevel = 1,
  pdfdisplaydoctitle = true }

\title[Categorification of \(\fraksl_k\)--representations]{Categorification of tensor product representations of $\fraksl_k$ and category \(\catO\)}
\author{Antonio Sartori}
\address{A. S.: Department of Mathematics\\%
University of York\\%
York, YO10 5DD (UK)}
\email{antonio.sartori@york.ac.uk}
\author{Catharina Stroppel}
\address{C.S.: Mathematikzentrum\\University of Bonn\\53115 Bonn (Germany)}
\email{stroppel@math.uni-bonn.de}
\date{}
\thanks{The first author has been supported by the EPSRC grant EP/I014071.}

\begin{document}

\begin{abstract}
  We construct categorifications of tensor products of arbitrary finite-dimensional irreducible representations of \(\fraksl_k\) with subquotient categories of the BGG category \(\catO\), generalizing previous work of Sussan and Mazorchuk-Stroppel.  Using Lie theoretical methods, we prove in detail  that they are tensor product categorifications according to the recent definition of Losev and Webster. As an application we deduce an equivalence of categories between certain versions of category  \(\catO\) and Webster's tensor product categories. Finally we indicate how the categorifications of tensor products of the natural representation of $\mathfrak{gl}(1|1)$ fit into this framework. 
\end{abstract}

\maketitle

\section{Introduction}
\label{sec:introduction}

Since  the groundbreaking work of Khovanov \cite{MR1740682}, substantial progress has been made in the categorification of  irreducible representations of Lie algebras and their tensor products. Milestones were the introduction  of the Khovanov-Lauda-Rouquier algebras (\cite{2008arXiv0812.5023R}, \cite{MR2525917}, \cite{MR2763732}, \cite{MR2628852}) and the establishment of existence and uniqueness results (\cite{MR2373155}, \cite{2013arXiv1303.1336L}). In \cite{2013arXiv1303.1336L}, Losev and Webster gave for the first time a formal definition of \emph{tensor product categorification}, with which they were able to prove a strong uniqueness result. They also showed that for each finite tensor product of finite dimensional irreducible representations of a complex semisimple Lie algebra such a categorification exists, using Webster's diagram algebras, \cite{2013arXiv1309.3796W}.

In type A, a big role in categorification has always been played  by the BGG category \(\catO\) \cite{MR0407097}. Categorifications of  tensor powers of the vector representation of \(\fraksl_2\) using category \(\catO\) have been constructed by Bernstein, Frenkel, Khovanov \cite{MR1714141} and Frenkel, Khovanov and the second author \cite{MR2305608}. Later, categorifications of fundamental representations of \(\fraksl_k\) for \(k \geq 2\) using parabolic subcategories of the BGG category \(\catO(\gl_\enne)\) have been constructed in \cite{2007math......1045S} and \cite{MR2567504}. In the present paper, we generalize their construction to arbitrary irreducible representations using subquotient categories of \(\catO(\gl_\enne)\). Moreover, we prove that this construction is a tensor product categorification according to \cite{2013arXiv1303.1336L}, and hence is equivalent to Webster's diagrammatic categorification.

We point out that it was already known to experts 
that subquotient categories of \(\catO(\gl_\enne)\) categorify
arbitrary tensor products of \(\fraksl_k\)--representations,
although details cannot be found in the literature. The
existence of our construction is in fact implied by
Webster's categorification
\cite{2013arXiv1309.3796W}. Indeed, as Webster proved, the
category \(\catO(\gl_\enne)\) is equivalent to the module
category over his diagram algebra. Since Webster's categorification
of arbitrary tensor products is obtained via idempotent
truncations and quotients of that diagram algebra, it
follows via this equivalence that the same can be done with
subquotient categories of \(\catO(\gl_\enne)\) (see
\cite[Proposition~8.8]{2013arXiv1309.3796W}).  The advantage
of the present paper, however, is that we identify these
categories explicitly inside \(\catO(\gl_\enne)\), and we
prove all details using only Lie theory, without making use
of the equivalence with Webster's diagram algebra. We hope
that this on the one hand can provide a better understanding of the involved
subquotient categories of \(\catO\) and of the Lie
theoretical categorification, and on the other hand provides the possibility to better understand the established Lie theory via the viewpoint of the diagram algebras.

In order to state our main result, let us introduce some notation.
Let \(\lambda^{(1)},\dotsc,\lambda^{(m)}\) be dominant integral weight for \(\fraksl_k\), and let 
\(V(\lambda^{(1)}),\dotsc,V(\lambda^{(m)})\) be the irreducible finite-dimensional \(\fraksl_k\)--representations with corresponding highest weight. There is a standard way to identify each of the \(\lambda^{(i)}\) with a partition of \(\enne_i\). Set also  \(\enne=\enne_1+\dotsb+\enne_m\). 
We summarize the  results of the paper in the following theorem (which should be compared with {\cite[Proposition~8.8]{2013arXiv1309.3796W}}):

\begin{theorem}
  \label{thm:5}
  There exists a subquotient category
  \(\calQ^{\lambda}_I\) of \(\catO(\gl_{\enne})\) together
  with endofunctors \(\sfE\), \(\sfF\) and an action of the
  KLR 2-category which defines an \(\fraksl_k\)--tensor
  product categorification of \(V(\boldlambda)=V(\lambda^{(1)}) \otimes
  \dotsb \otimes V(\lambda^{(m)})\). This can be lifted to a
  graded \(U_q(\fraksl_k)\)--categorification.
\end{theorem}

We should point out that, unfortunately, the Lie theoretical setting for categorification becomes slightly less pleasant when we want to discuss \emph{graded} categorifications of \(U_q(\fraksl_k)\)--representations, since we are still not able to prove that we have a graded categorical action using only Lie theory. Indeed, we have a graded version of category \(\catO\) (see \cite{MR1322847} and \cite{MR2005290}), which induces a grading also on the subquotient categories, and we also have graded lifts of all the functors, but we still miss a direct proof that the action of the degenerate affine Hecke algebra via natural transformations on translation functors lifts to a graded action of the KLR-algebra, expect for special cases treated in \cite{MR2781018}. As Webster explains in \cite[Section~8]{2013arXiv1309.3796W}, this can be deduced from the uniqueness of the Koszul grading. Appealing to this fact, we will discuss the graded categorification at the end of the paper.

We remark that subquotient categories of \(\catO(\gl_\enne)\) were introduced already in \cite{MR1921761}, where they were called \emph{generalized parabolic subcategories}; regular blocks have been used in \cite{MR2450613} to categorify induced Hecke modules, while regular and singular blocks, but only in some special cases, have been used in \cite{miophd2} to categorify representation of the general Lie superalgebra \(\gl(1|1)\). In the present construction, regular and singular blocks appear in full generality. We believe that the categorification result of this paper provides a better understanding of such categories, in particular thanks to the following direct consequence which provides a diagrammatical description of these categories.

\begin{corollary}
\label{Cor1}
The category \(\calQ^{\lambda}_I\) categorifying \(V(\boldlambda)\) is equivalent to modules over Webster's tensor algebra, \cite{2013arXiv1309.3796W} attached to \(\boldlambda\).   
\end{corollary}
This result follows directly from our main theorem and the uniqueness result of \cite{2013arXiv1303.1336L}.  The same result was proved already in  \cite[Proposition 8.8]{2013arXiv1309.3796W} directly, avoiding the power of the uniqueness result (with the distadvantage of having to check many technical details).  

Let $A=A_n$ be the (graded) diagram algebra introduced in \cite{miophd2} and \cite{Sartori} such that  the category $C(n)$ of (graded) $A_n$--modules categorifies, in the weak sense, the $n$-fold tensor products of the (quantized) natural representation of \(\gl(1|1)\). Then we obtain as a consequence of the main theorem:
\begin{corollary}
\label{Cor2}
The category \(C(n)\) is a Serre subquotient of the (graded) category \(\calQ^{\lambda}_I\) which defines a (graded)  \(\fraksl_k\)--tensor
  product categorification of  \(V(\boldlambda)=V(\lambda^{(1)}) \otimes\dotsb \otimes V(\lambda^{(n)})\), where \(V(\lambda^{(i)})=V(\varpi_1)\), for \(1\leq i\leq n\), is the first fundamental representation of (quantum) \(\fraksl_k\). 
\end{corollary}

Let us now briefly discuss the main idea of the construction of the categorification, which is more or less implicit in \cite[Section~8]{2013arXiv1309.3796W}. An easy but important observation is that if \(\varpi_1,\dotsc,\varpi_r\) denote the fundamental weights of \(\fraksl_k\) and \(\lambda=a_1 \varpi_1 + \dotsb + a_{k-1} \varpi_{k-1}\) for \(a_j \in \Z_{\geq 0}\) an integral dominant weight, then we have an embedding of representations
\begin{equation}\label{eq:3}
    \begin{tikzpicture}[baseline=(current bounding box.center)]
  \node (A) at (0,0) {  $V(\lambda)$}; 
  \node (B) at (3,0) { $\displaystyle \bigotimes_{r=1}^{k-1} V(\varpi_r)^{\otimes a_r} $};
  \node (C) at (6,0) { $\displaystyle \bigotimes_{l=1}^{\sum_r a_r r} V,$};
  \path[right hook->] (A) edge (B);
  \path[right hook->] (B) edge (C);
\end{tikzpicture}
\end{equation}
where \(V(\lambda)\) is the irreducible (finite-dimensional) \(\fraksl_k\)--representation of highest weight \(\lambda\) and \(V=V(\varpi_1)\) is the vector representation. If we consider \(\lambda\) as a partition and let \(\enne=\abs{\lambda}=\sum_r a_r r\), then the categorification of \cref{eq:3} becomes
\begin{equation}\label{eq:112}
    \begin{tikzpicture}[baseline=(current bounding box.center)]
  \node (A) at (0,0) {  $\calQ^\lambda_I$}; 
  \node (B) at (3,0) { $\catO^\lambda(\gl_\enne)_I$};
  \node (C) at (6,0) { $\catO(\gl_\enne)_I,$};
  \path[<<-] (A) edge (B);
  \path[right hook->] (B) edge (C);
\end{tikzpicture}
\end{equation}
where \(\catO^\lambda(\gl_\enne)\) denotes parabolic category \(\catO\) attached to the standard parabolic subalgebra of type \(\gl_{\lambda_1} \oplus \dotsb \oplus \gl_{\lambda_\ell}\). Here the subscript \(I=\{1,\dotsc,k\}\) denotes the restriction to the blocks where (shifted) highest weights are sequences in \(I^n\).

Now, if \(\boldlambda=(\lambda^{(1)},\dotsc,\lambda^{(m)})\) is a sequence of integral dominant weights for \(\fraksl_k\), then the outer tensor product \(\calQ^{\lambda^{(1)}}_I \boxtimes \dotsb \boxtimes \calQ^{\lambda^{(m)}}_I\) gives a categorification of \(V(\boldlambda)=V(\lambda^{(1)}) \otimes \dotsb \otimes V(\lambda^{(m)})\) as an \(\fraksl_k^{\oplus m}\)--categorification. In order to get an \(\fraksl_k\) tensor product categorification, we need to glue the blocks of this categorification together. From a general and abstract point of view, it is a very interesting and challenging problem how this gluing should work. In our setting, we have at our disposal all the power of the BGG category \(\catO\), and it turns out that  the gluing procedure is given by parabolic induction; in some special cases, this is already visible in \cite{MR1714141} and \cite{MR2305608}. In detail, if \(\enne_i = \abs{\lambda^{(i)}}\) then \(\calQ^{\lambda^{(1)}}_I \boxtimes \dotsb \boxtimes \calQ^{\lambda^{(m)}}_I\) can be identified with a subcategory of \(\catO(\frakl)\), where \(\frakl = \gl_{\enne_1} \oplus \dotsb \oplus \gl_{\enne_m}\). The categorification of \(V(\boldlambda)\) is then essentially the image of the parabolic induction \(\Delta = U(\gl_\enne) \otimes_\frakp \blank\) inside \(\catO(\gl_\enne)\), where \(\enne=\enne_1+\dotsb+\enne_m\) and \(\frakp \subseteq \gl_\enne\) is the standard parabolic subalgebra with Levi factor \(\frakl\). We believe that this construction can give a better insight into tensor product categorifications from an abstract point of view.

\subsection*{Structure of the paper}

In \Cref{sec:comb-repr} we recall basic results on finite-dimensional representation theory of \(\fraksl_k\) and of its quantum enveloping algebra \(U(\fraksl_k)\).  \Cref{sec:combinatorics} fixes our conventions for  partitions and Young tableaux, while \Cref{sec:serre-subc-serre} is dedicated to basic facts about Schurian categories, Serre subcategories, quotient categories and standardly stratified categories.
In \Cref{sec:gener-parab-categ}, the technical heart of the paper, we will define the relevant subquotient categories of \(\catO\). We will also define the standardization functor \(\Delta\) via parabolic induction, and we will prove that it defines a standardly stratified structure. In \Cref{sec:acti-degen-affine} we will recall from \cite{MR1652134} the action of the degenerate affine Hecke algebra on translation functors of \(\catO\) and we will define categorical \(\fraksl_k\)--actions according to \cite{2008arXiv0812.5023R} and \cite{2013arXiv1303.1336L}. Using all this machinery we discuss finally in \Cref{sec:categ-simple-repr,sec:categ-tens-prod} the categorification of simple \(\fraksl_k\)--representations and tensor products of simple representations, respectively. Finally, in order to lift our categorification to an \(U_q(\fraksl_k)\)--one, we recall in \Cref{sec:graded-lift-soergel} the basics about graded category \(\catO\) and discuss in \Cref{sec:grad-categ-1} graded categorifications.

\subsubsection*{Acknowledgements}
\label{sec:acknowledgements}

The authors would like to thank Ben Webster for useful explanations on his categorification, grading and relations with category \(\catO\). 

\section{\texorpdfstring{\(U_q(\fraksl_k)\)}{Uq(sl(k))} and its representations}
\label{sec:comb-repr}

Throughout  the whole paper the base field is $\mathbb{C}$ and we fix a positive integer \(k\) and set \(I=\{1,\dotsc,k\}\). We denote by \(\fraksl_k=\fraksl_k(\C)\) the simple Lie algebra of all traceless \(k \times k\) complex matrices.

\subsection{Representation theory of \texorpdfstring{\(\fraksl_k\)}{sl(k)}}
\label{sec:repr-theory-fraksl_k}

Let \(\tilde \frakd\) denote the diagonal matrices and \(\frakd = \tilde \frakd \cap \fraksl_k\) the traceless diagonal matrices. If \(\tilde\updelta_1,\dotsc,\tilde\updelta_k\) is the standard basis of \(\tilde \frakd^*\), dual to the standard basis of monomial matrices of \(\tilde\frakd\), we denote by \(\updelta_1,\dotsc,\updelta_k\) the restrictions of \(\tilde \updelta_1,\dotsc,\tilde \updelta_k\) to \(\frakd^*\). (We use this special notation for \(\fraksl_k\) because we want to keep the usual notation for the Lie algebra \(\gl_\enne\), which will appear in \Cref{sec:gener-parab-categ}.)

Let \(\alpha_i=\updelta_i - \updelta_{i+1}\) for \(i=1,\dotsc,k-1\). Let $\Pi=\{\alpha_1,\ldots,\alpha_{k-1}\} \subset \frakd^*$ be the set of simple roots. Let $(\cdot,\cdot)$ denote the standard symmetric non-degenerate bilinear form on $\mathfrak d^*$.
 The simple roots satisfy
\begin{equation}
  \label{eq:1}
  (\alpha_i,\alpha_j) = a_{ij} = 
  \begin{cases}
    2 & \text{if } i=j,\\
    -1 & \text{if } \abs{i-j}=1,\\
    0 & \text{if } \abs{i-j}>1.
  \end{cases}
\end{equation}
Let $\{\varpi_i \suchthat i =1,\ldots,k-1\}\subset \mathfrak d^*$ be the dual basis to $\Pi$ with respect to the bilinear form. Let $\weightssl=\bigoplus_{i=1}^{k-1} \Z \varpi_i$ be the weight lattice and $\weightssl^+ = \bigoplus_{i=1}^{k-1} \Z_{\geq 0} \varpi_i$ be the set of integral dominant weights. By definition, an \emph{integral dominant  weight} \(\varpi\) is an integral combination \(\varpi=c_1 \varpi_1 + \dotsb + c_{k-1} \varpi_{k-1}\) of fundamental weights with non-negative coefficients \(c_i \in \Z_{\geq 0}\). 

As well-known, the category of finite-dimensional representations of \(\fraksl_k\) is semisimple, and the isomorphism classes of simple objects are in bijection with \(\weightssl^+\) by taking highest weights. For \(\lambda \in \weightssl^+\) we denote by \(V(\lambda)\) the irreducible \(\fraksl_k\)--module with highest weight \(\lambda\).

\begin{example}
  The representation \(V=V(\varpi_1)=\C^k\) is  the \emph{vector
  representation} or \emph{natural representation} and a special example of the \emph{fundamental representations}
  \(V(\varpi_i)\cong \bigwedge^iV(\varpi_1)\)  for \(i=1,\dotsc,k-1\) . \label{ex:2}
\end{example}

Given an arbitrary \(  \lambda= c_1 \varpi_1 + \cdots + c_{k-1} \varpi_{k-1} \in \weightssl^+\)
set
\begin{equation}
  \label{eq:5}
  \widetilde V(\lambda) = \bigotimes_{r=1}^{k-1} V(\varpi_r)^{\otimes c_r}.
\end{equation}

Since the highest weight of \(\widetilde V(\lambda)\) is \(\lambda\),  there is an embedding of \(V(\lambda)\) into \(\widetilde V(\lambda)\). Hence
\begin{equation}
  \label{eq:31}
  V(\lambda) \hookrightarrow \widetilde V(\lambda) \hookrightarrow V^{\otimes \sum_r c_r r}.
\end{equation}
It follows in particular that each finite-dimensional \(\fraksl_k\)--representation is a subrepresentation of a tensor power of the vector representation.

\subsection{The quantum group \texorpdfstring{\(\Uqslk\)}{Uq(sl(k))}}

The {\em quantum group} $\Uqslk$ is the
Hopf algebra over $\C(q)$ generated by $\{E_i,F_i,K_i,K_i^{-1} \suchthat i=1,\ldots,k-1\}$ subject to some relations which we do not want to recall here (see for example \cite{MR2628852}).

We restrict ourselves to type I finite-dimensional representations of $\Uqslk$, that is weight representations \(W= \bigoplus W_\mu\) with
\begin{equation}
W_\mu = \{ v_\mu \in W \suchthat K_i v_\mu = q^{(\alpha_i, \mu)} v_\mu \text{ for all } i =1,\dotsc,k-1\}.\label{eq:102}
\end{equation}
Then the finite-dimensional representation theory of \(U_q(\fraksl_k)\) is analogous to the one of \(\fraksl_k\). In particular, such finite-dimensional representations are semisimple with the irreducible ones parametrized by their highest weights (see \cite{MR1359532} or \cite{MR2759715}). For \(\lambda \in \weightssl^+\). We denote by \(V_q(\lambda)\) be the simple \(U_q(\fraksl_k)\)--module of highest weight \(\lambda\).
\begin{example}
  \label{ex:3}
  The $k$--dimensional vector representation of $\Uqslk$ is the
  representation $V_q=V_q(\varpi_1)$. Explicitly, $V_q$ has
  standard basis $\{v_1,\ldots,v_k\}$ over $\C(q)$, and the action of
  $\Uqslk$ is given by
  \begin{equation}
    \label{eq:115}
    E_i v_j  = \delta_{i+1,j} v_{i}, \qquad
    F_i v_j  = \delta_{ij} v_{i+1}, \qquad
    K_i^{\pm1} v_j = q^{\pm(\delta_{ij}- \delta_{i+1,j})} v_{j},
  \end{equation}
  where \(\delta_{ij}\) is the Kronecker function. Let \(\bigwedge_q^r V_q\) be the subspace of \(\bigotimes^r V_q\) on basis
  \begin{equation}
    \label{eq:12}
    v_{j_1} \wedge \cdots \wedge v_{j_r} = \sum_{\sigma \in \bbS_r} (-1)^{\sgn(\sigma)} q^{\len(\sigma)} v_{j_{\sigma(1)}} \otimes \cdots \otimes v_{j_{\sigma(r)}}
  \end{equation}
for \(k \geq j_1 \geq \dotsb \geq j_r \geq 1\). One can check that   the action of $\Uqslk$ on \cref{eq:12} is given by the formulas
  \begin{equation}
    \label{eq:47}
    \begin{aligned}
      E_i (v_{j_1} \wedge \cdots \wedge v_{j_r})  & =
      \begin{cases}
        v_{j_1} \wedge \cdots \wedge v_{j_h - 1} \wedge \cdots \wedge v_{j_r} & \text{if some } j_h=i+1\\
        0 & \text{otherwise},
      \end{cases}\\
      F_i (v_{j_1} \wedge \cdots \wedge v_{j_r})  & =
      \begin{cases}
        v_{j_1} \wedge \cdots \wedge v_{j_h + 1} \wedge \cdots \wedge v_{j_r} & \text{if some } j_h=i\\
        0 & \text{otherwise},
      \end{cases}\\
      K_i^{\pm 1} (v_{j_1} \wedge \cdots \wedge v_{j_r})  & = q^{\pm \big( \sum_{h=1}^r \delta_{i, j_h}- \delta_{i+1,j_h} \big)}\, v_{j_1} \wedge \cdots \wedge v_{j_r},
    \end{aligned}
\end{equation}
and hence \(\bigwedge_q^r V_q\) is a \(U_q(\fraksl_k)\)--subrepresentation. It is irreducible and has highest weight \(\varpi_r\). Hence \(\bigwedge_q^r V_q\) is isomorphic to \(V(\varpi_r)\), and is called the \emph{\(r\)--th fundamental representation} of \(\Uqslk\). (For the general definition of the exterior power in the quantized setting, see for example \cite{MR2386232} or \cite{2012arXiv1210.6437C}.)
 \end{example}

Given an arbitrary \(
  \lambda= c_1 \varpi_1 + \cdots + c_{k-1} \varpi_{k-1} \in \weightssl^+\), analogously to \cref{eq:5},
set
\begin{equation}
\label{eq:113}
  \widetilde V_q(\lambda) = \bigotimes_{r=1}^{k-1} V_q(\varpi_r)^{\otimes c_r}.
\end{equation}
As in the non-quantized case, there is an embedding
\(  V_q(\lambda) \hookrightarrow \widetilde V_q(\lambda) \hookrightarrow V_q^{\otimes \sum_r c_r r}\).

\section{Combinatorics of partitions and tableaux}
\label{sec:combinatorics}

\subsection{Partitions and Young diagrams}
\label{sec:part-young-diagr}

A \emph{partition} \(\lambda\) of \(\enne\) is a non-increasing sequence \((\lambda_1,\dotsc,\lambda_\ell)\) of positive numbers with \(\lambda_1+\dotsc+\lambda_\ell=\enne\) for some \(\ell \geq 0\). We write \(\abs{\lambda}=\enne\). We denote by \(\Part(\enne)\) the set of partitions of \(\enne\). To a partition we associate a Young diagram, which we also denote by \(\lambda\), as in the following picture:
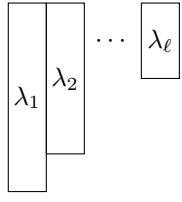
\begin{equation*}
  \begin{tikzpicture}[scale=0.5,x={(0,-1cm)},y={(-1cm,0)}]
    \draw (0,0) rectangle (5,-1);
    \draw (0,-1) rectangle (4,-2);
    \draw (0,-3.5) rectangle (2,-4.5);
    \node at (1,-2.75) {$\cdots$};
    \node at (2.5,-0.5) {$\lambda_1$};
    \node at (2,-1.5) {$\lambda_2$};
    \node at (1,-4) {$\lambda_\ell$};
  \end{tikzpicture}
\end{equation*}
For example, if \(\lambda=(4,4,2,1,1)\) then the corresponding Young diagram is 
  \begin{equation*}
    \ytableausetup{centertableaux}
  \ydiagram{5,3,2,2}
 \end{equation*}

The transposed \(\lambda^T\) of a partition corresponds to the Young diagram reflected across the anti-diagonal. For example, the transposed of \((4,4,2,1,1)\) is \((5,3,2,2)\).

\begin{remark}
   \label{rem:4}
   To match the combinatorics of the category \(\catO\), our convention is transposed to the usual one in the literature.
 \end{remark}

Given  \(\varpi= c_1 \varpi_1 + \dotsb + c_{k-1} \varpi_{k-1} \in \weightssl^+\) we set \(\lambda^T_j=c_j+\dotsb+c_{k-1}\) for \(j=1,\dotsc,k-1\) and so associate to \(\varpi\) a partition \(\lambda\) with at most \(k-1\) rows. 
Graphically:
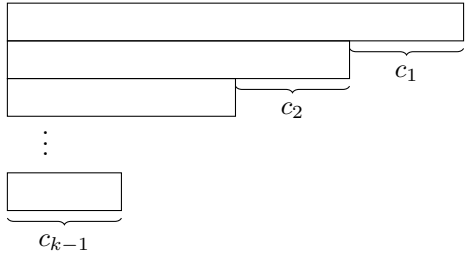
\begin{equation*}
  \begin{tikzpicture}[scale=0.5]
    \draw (0,0) rectangle (12,-1);
    \draw (0,-1) rectangle (9,-2);
    \draw (0,-2) rectangle (6,-3);
    \draw (0,-4.5) rectangle (3,-5.5);
    \node at (1,-3.5) {$\vdots$};
    \draw [yshift=-0.2cm,decorate,decoration={brace,mirror}] (9,-1) -- node[below=0.1cm] {$c_1$} ++(3,0);
    \draw [yshift=-0.2cm,decorate,decoration={brace,mirror}] (6,-2) -- node[below=0.1cm] {$c_2$} ++(3,0);
    \draw [yshift=-0.2cm,decorate,decoration={brace,mirror}] (0,-5.5) -- node[below=0.1cm] {$c_{k-1}$} ++(3,0);
  \end{tikzpicture}
\end{equation*}
This defines a bijection between the set of integral dominant weights for \(\mathfrak{sl}_k\) and the set of partitions with at most \(k-1\) rows. From now on, we will just identify them.

\subsection{Tableaux}
\label{sec:tableaux-1}

A \emph{tableau} of \emph{shape} \(\lambda\) is obtained by filling the boxes of a Young diagram \(\lambda\) with integer numbers. A tableau is
\begin{itemize}
\item \emph{column-strict} if the entries are
  strictly increasing along the columns,
\item \emph{semi-standard} if it is column-strict and the entries are non-decreasing along the rows,
\item \emph{standard} if it is column-strict and the entries are strictly increasing along the rows.
\end{itemize}

The (reversed column) \emph{reading word} \(\bolda\) attached to a tableau is obtained by reading columnwise from the left and bottom. For example, by reading the tableau
  \begin{equation}\label{eq:109}
    \ytableausetup{centertableaux}
   \ytableaushort{367,25,14}
  \end{equation}
we obtain the sequence \(\bolda=(1,2,3,4,5,6,7)\). Conversely, given a sequence \(\bolda=(a_1,\dotsc,a_\enne)\) of integer numbers, we denote by \(\tableau^\lambda(\bolda)\) the tableau of shape \(\lambda\) obtained by filling the boxes of \(\lambda\) with the numbers of \(\bolda\) first along the columns and then along the rows, starting from the bottom left corner. For example, if \(\lambda=(3,3,1)\)  and \(\bolda=(1,2,3,4,5,6,7)\) then \cref{eq:109} is the tableau \(T^\lambda(\bolda)\).

We denote by \(\St^\lambda(I)\) the set of semi-standard tableaux with entries \(I\). Recall the following well-known fact (see e.g.\ \cite[Section~5]{MR1357199} or \cite[Chapter~8]{MR1464693}):

\begin{lemma}
  \label{lem:22}
  The dimension of \(V(\lambda)\) is equal to the cardinality of \(\St^\lambda(I)\).
\end{lemma}

Let \(\lambda\) be a partition of \(\enne\) and  \(\boldd=(d_1,\dotsc,d_\enne)\) a sequence of integers. A tableau of shape \(\lambda\) and type \(\boldd\) is a tableau \(T^\lambda(\bolda)\), where \(\bolda\) is a permutation of \(\boldd\). We denote by \(\St^\lambda(\boldd)\) the set of semi-standard tableaux of shape \(\lambda\) and type \(\boldd\).

\subsection{Multipartitions and multitableaux}
\label{sec:mult-mult}

A \emph{multipartition} of \(\enne\) is a sequence \(\boldlambda =(\lambda^{(1)},\dotsc,\lambda^{(m)})\) of partitions with \(\abs{\lambda^{(1)}}+\dotsb+\abs{\lambda^{(m)}}=\enne\). We denote by \(\Part^m(\enne)\) the set of multipartitions of \(\enne\) with \(m\) parts. A multitableau of shape \(\boldlambda\) is a sequence \(\boldT=(T^{(1)},\dotsc,T^{(m)})\) of tableaux such that \(T^{(j)}\) is of shape \(\lambda^{(j)}\).  It is called column-strict (respectively, semi-standard or standard) if all the \(T^{(j)}\) are column-strict (respectively, semi-standard or standard).

The (reversed column) reading word attached to a multitableau  is the concatenation of the reading words attached to the single tableaux. If \(\boldlambda\) is a multipartition then we denote by \(\tableau^\boldlambda(\bolda)\) the multitableau of shape \(\boldlambda\) obtained by filling the boxes of the partitions of \(\boldlambda\) with the entries of \(\bolda\), starting with the Young diagram of \(\lambda^{(1)}\).

As before, \(\St^\boldlambda(I)\) denotes the set of semi-standard multitableaux of shape \(\boldlambda\) filled with entries from \(1\) to \(k\). If \(\boldd=(d_1,\dotsc,d_n)\) then a multitableau of shape \(\boldlambda\) and type \(\boldd\) is a multitableau \(T^\boldlambda(\bolda)\), where \(\bolda\) is a permutation of \(\boldd\). We denote by  \(\St^\boldlambda(\boldd)\) the set of semi-standard multitableaux of shape \(\boldlambda \) and type \(\boldd\).

\section{Preliminaries on category theory}
\label{sec:serre-subc-serre}

We will mostly denote categories by calligraphic letters like \(\calA, \calB, \ocalC\) and functors by capital letters in \emph{sans-serif}, like \(\sfF, \sfG, \sfT\). 

Let \(\calA\) be an abelian \(\C\)--linear category. Then \(\calA\) is said to be \emph{finite} if 
\begin{enumerate*}[(i)]
  \item the homomorphism spaces are finite dimensional,
  \item all objects are of finite length, and
  \item there are only finitely  many simple objects up to isomorphism, each of which has a projective cover.
\end{enumerate*}
We always assume that the endomorphism algebras of simple objects are one-dimensional.
These requirement suffice to ensure that \(\calA\) is equivalent to the category of finite-dimensional modules over a finite-dimensional \(\C\)--algebra. 

All our categories will be direct sums of finite abelian categories,
 or in other words their blocks will be finite abelian categories. 
The constructions and results of this section, which we explain and state for finite abelian categories only, apply directly to their direct sums by considering blocks.

We will denote by \([\calA]\) the complexified Grothendieck group of \(\calA\). In particular,   \([\Vect]\cong\mathbb{C}\) for the category \(\Vect\) of finite-dimensional complex vector spaces. 

\subsection{Serre subcategories and Serre quotient categories}
\label{sec:serre-subc-serre-1}

Let \(\calA\) be a finite abelian category, and let \(\{L(\lambda) \suchthat \lambda \in \Lambda\}\) be (a complete set of representatives for the isomorphism classes of) the simple objects of \(\calA\).  Let \(P(\lambda) \) be the projective cover of \(L(\lambda)\), and let \(A=\End\big(\bigoplus_{\lambda \in \Lambda}P(\lambda)\big)\). Then \(\calA \cong \rmod{A}\).

Given \(\Gamma \subseteq \Lambda\), the \emph{Serre subcategory} \(\calS_\Gamma\) is the full subcategory of \(\calA\) consisting of object with composition factors of the type \(L(\gamma)\) for \(\gamma \in \Gamma\). For \(M \in \calA\) let \(\sfZ_\Gamma (M)\) be the maximal quotient of \(M\) lying in \(\calS_\Gamma\). It is easy to show that this defines a functor \(\sfZ_\Gamma\colon \calA \mapto \calS_\Gamma\), which we call  \emph{Zuckerman functor}. This is left adjoint to the inclusion functor, and hence is right exact.

Let \(\calA/\calS_\Gamma\) denote the Serre quotient (see \cite{MR0232821}). It is an abelian category and comes with an exact quotient functor \(\sfC_\Gamma\colon \calC \mapto \calC/\calS_\Gamma\), which we call \emph{coapproximation functor}. We have \(\sfC_\Gamma(L(\lambda))=0\) if and only if \(\lambda \in \Gamma\), and \(\{\sfC_\Gamma L(\lambda) \suchthat \lambda \in \Lambda -\Gamma\}\) gives a full set of simple objects in \(\calC/\calS_\Gamma\) up to isomorphism. We have an equivalence of categories \(\calC/\calS_\Gamma \cong \rmod{\End\big( \bigoplus_{\lambda \in \Lambda - \Gamma} P(\lambda)\big)}\),  see \cite[Proposition~33]{MR2774639}. Notice that \(\End\big( \bigoplus_{\lambda \in \Lambda - \Gamma} P(\lambda)\big)\) is an idempotent truncation \(eAe\) of the algebra \(A=\End\big(\bigoplus_{\lambda \in \Lambda}P(\lambda)\big)\) for some idempotent \(e \in A\), and the quotient or coapproximation functor is given by \(\Hom_A(eA, \blank)\).

Given a projective object \(P \in \calA\) we denote by \(\operatorname{Add}(P)\) the additive subcategory of \(\calA\) additively generated by direct summands of \(P\). An object \(M \in \calA\) is  \(\operatorname{Add}(P)\)--\emph {presentable} if it has a presentation \(P \mapto Q \surto M\) with \(P,Q \in \operatorname{Add}(P)\).
 Then the category \(\calC/\calS_\Gamma\) can be identified with the full subcategory of \(\calC\) consisting of \(\operatorname{Add}\big(\bigoplus_{\lambda \in \Lambda - \Gamma} P(\lambda)\big)\)--presentable objects. The inclusion of \(\calC/\calS_\Gamma\) in \(\calC\) is not exact, but is right exact, since it is left adjoint to the coapproximation functor (see \cite{MR2139933}). Under the equivalence above, it is given by \(\blank \otimes_{eAe} eA\colon \rmod{eAe} \longrightarrow \rmod{A}\).

\subsection{Standardly stratified categories}
\label{sec:stand-strat-categ}

We recall the definition of a standardly stratified category, following \cite{2013arXiv1303.1336L}. Let \(\calA\) be as above and let \(\Xi\) be a poset with a map \(p\colon \Lambda \mapto \Xi\). For \(\lambda \in p^{-1}(\xi)\) we denote by \(L^\xi(\lambda)\) the simple object in \(\calA_{\sim \xi}\) corresponding to \(\lambda\) and by \(P^\xi(\lambda)\) its projective cover. 

For \(\xi \in \Xi\) we denote by \(\calA_{\preleq \xi}\) the Serre subcategory \(\calS_\Gamma\) where \(\Gamma=\{\lambda \suchthat p(\lambda) \preleq \xi\}\) and by \(\calA_{\prel \xi}\) the Serre subcategory \(\calS_{\Gamma'}\) with \(\Gamma' = \{ \lambda \suchthat p(\lambda) \prel \xi\}\).
Moreover, we let \(\calA_{\sim \xi}\) be the Serre quotient \(\calA_{\preleq \xi}/\calA_{\prel \xi}\) with quotient functor  \(\pi_\xi\colon \calA_{\preleq \xi} \mapto \calA_{ \sim \xi}\).

We suppose that the functor \(\pi_\xi\) has an exact left-adjoint functor, which we call \emph{standardization functor} and which we denote by \(\Delta_\xi\). We set \(\Delta(\lambda)=\Delta_\xi(P^\xi(\lambda))\) and \(\oDelta(\lambda)=\Delta_\xi(L^\xi(\lambda))\). These objects are called the \emph{standard} and \emph{proper standard} module corresponding to \(\lambda\), respectively.

\begin{definition}
  \label{def:16}
  The category \(\calA\) together with the poset \(\Xi\) and the map \(p\colon \Lambda \mapto \Xi\) is called a \emph{standardly stratified category} if for all \(\lambda\) there is an epimorphism \(P(\lambda) \surto \Delta(\lambda)\) whose kernel admits a filtration by objects \(\Delta(\mu)\) with \(p(\mu ) \preg p(\lambda)\).
\end{definition}

If moreover  \(\calA_{\sim \xi}\) is equivalent to the category of vector spaces for each \(\xi \in \Xi\), then \(\calA\) is \emph{quasi-hereditary}. We will call \(\gr \calA = \bigoplus_{\xi \in \Xi} \calA_{\sim \xi}\) the \emph{associated graded category}.

\subsection{Outer tensor product of categories}
\label{sec:outer-tensor-product}

Let \(\calA\) and \(\calB\) be finite abelian categories. According to \cite[Section~5]{MR1106898} (cf.\ also \cite[\textsection{}1.46]{Etingof_finitetensor}), their \emph{outer tensor product} \(\calA \boxtimes \calB\) is defined and  comes along with a bifunctor \(\boxtimes \colon \calA \times \calB \mapto \calA \boxtimes \calB\) which is exact in both variables and satisfies
\begin{equation}
  \label{eq:111}
  \Hom_{\calA}(M_1,M_2) \otimes \Hom_{\calB}(N_1,N_2) \cong \Hom_{\calA \boxtimes \calB}(M_1 \boxtimes N_1, M_2 \boxtimes N_2). 
\end{equation}
If \(\calA = \rmod{A}\) and \(\calB = \rmod{B}\) where \(A\) and \(B\) are finite-dimensional \(\C\)--algebras, then \(A \otimes B\) is also a finite-dimensional algebra and
\begin{equation}
  \label{eq:110}
  \rmod{A} \boxtimes \rmod{B} \cong \rmod{A \otimes B}.
\end{equation}
This implies in particular that  \(\calA \boxtimes \calB\) is again a  finite abelian category. 

Observe that this also provides an isomorphism \([\rmod{A}] \otimes [\rmod{B}] \cong [\rmod{A} \boxtimes \rmod{B}] \) via \([M] \otimes [N] \mapto [M \otimes N]\), since the simple \((A\otimes B)\)--modules are precisely the outer tensor products of a simple \(A\)--module and a simple \(B\)--module.

\section{Category \texorpdfstring{\(\catO\)}{O} and subquotient categories}
\label{sec:gener-parab-categ}

Let us fix a positive integer \(\enne\). Let \(\gl_\enne=\gl_\enne(\C)\) be
the general Lie algebra of \(\enne \times \enne\) matrices with the
standard Cartan decomposition \(\gl_\enne=\frakn^- \oplus \frakh
\oplus \frakn^+\) into strictly lower diagonal, diagonal and strictly upper diagonal matrices respectively. Let \(\frakb = \frakh \oplus \frakn^+\)
be the standard Borel subalgebra. We let \(\epsilon_1,\dotsc,\epsilon_\enne\) be the basis of \(\frakh^*\) dual to the standard basis of monomial diagonal matrices, and set
\begin{equation}
  \label{eq:62}
  \rhogl=-\epsilon_2 -2 \epsilon_3 - \dotsb - (\enne -1) \epsilon_\enne.
\end{equation}
Let  \(\weightsgl \subset \frakh^*\) denote the set of integral weights, and \(W=\bbS_n\) the Weyl group of \(\gl_n\).

The choice of basis \(\epsilon_1,\dotsc,\epsilon_\enne\) defines an isomorphism \(\frakh^* \cong \C^\enne\), which restricts to a bijection \(     a_1 \epsilon_1 + \dotsc + a_\enne \epsilon_\enne  \mapsto (a_1,a_2,\dotsc,a_\enne)\) between \(\weightsgl\) and \(\Z^\enne\).
From now on, we will identify \(\weightsgl\) with \(\Z^\enne\) and denote elements of \(\weightsgl\) by bold roman letters, like \(\bolda,\boldb,\boldd\). 

\subsection{The category \texorpdfstring{$\catO$}{O}}
\label{sec:category-cato}

We recall now some basic facts on the BGG category \(\catO\). For more details see \cite{MR2428237}.

\begin{definition}[\cite{MR0407097}]
  The integral BGG category \(\catO=\catO(\gl_\enne)=\catO(\gl_\enne;\frakb)\)
  is the full subcategory of
  \(U(\gl_\enne)\)--modules which are
  \begin{enumerate}[(O1)]
  \item\label{item:30} finitely generated as \(U(\gl_\enne)\)--modules,
  \item\label{item:31} weight modules for the action of \(\frakh\) with
    \emph{integral} weights, and
  \item\label{item:32} locally \(\frakn^+\)--finite.
  \end{enumerate}\label{def:6}
\end{definition}

We stress that we consider here only modules with integral
weights. The category \(\catO\) is Schurian  (i.e. abelian, \(\C\)--linear with enough projective and injective objects, such that all objects are of finite length and the endomorphism algebras of irreducible objects are one dimensional), and
it is obviously closed under tensoring with finite dimensional
\(\gl_\enne\)--modules.

For \(\bolda \in \weightsgl\) we denote by
\(M(\bolda) \in \catO\) the Verma module with highest weight
\(\bolda - \rhogl \) (e.g. \(M(0,\dotsc,0)\) is the most singular Verma module with highest weight \(-\rhogl\)). Let \(L(\bolda)\) denote its unique
simple quotient and \(P(\bolda)\) its projective cover. We let \(\weightsgl^+=\weightsgl^+(\gl_n)\) be the set of (shifted) integral dominant weights:
\begin{equation}
  \label{eq:64}
  \weightsgl^+(\gl_n) = \{\bolda=(a_1,\dotsc,a_\enne) \in \weightsgl \suchthat a_1 > a_2 > \dotsb > a_\enne\}.
\end{equation}
Recall that \(L(\bolda)\) is finite dimensional if and only if \(\bolda \in \weightsgl^+\).

We denote by \((w,\bolda) \mapsto w  \bolda\) 
the standard action of \(\bbS_\enne\) on \( \weightsgl=\Z^\enne\) by permutations. Then for  \(\boldd \in \weightsgl^+\), the Serre subcategory of \(\catO\) generated by \(L(w\boldd)\) for \(w \in W\) forms a block of \(\catO\) which we denote by \(\catO_\boldd\). Hence we have a block decomposition
\begin{equation}
  \label{eq:98}
  \catO = \textstyle\bigoplus_{\boldd \in \weightsgl^+} \catO_\boldd.
\end{equation}

\subsection{The parabolic category \texorpdfstring{\(\catO^\tau\)}{O}}
\label{sec:parab-categ-cato}

Let now \(\tau=(\tau_1,\dotsc,\tau_m)\) be a composition of \(\enne\) and \(\frakp_\tau \subseteq \gl_\enne\) the associated standard parabolic subalgebra with Levi factor \(\gl_\tau=\gl_{\tau_1} \times \dotsb \times \gl_{\tau_m}\) and nilpotent part \(\fraku_\tau\), so that \(\frakp_\tau = \gl_\tau \oplus \fraku_\tau\).

\begin{definition}
  The parabolic category $\catO^\tau = \catO^{\frakp_\tau}$
  is the full subcategory of
  \(U(\gl_\enne)\)--modules which are
  \begin{enumerate}[(OP1)]
  \item\label{item:27} finitely generated as \(U(\gl_\enne)\)--modules,
  \item\label{item:28} a direct sum of finite-dimensional simple modules as \(\gl_\tau\)--modules, with
    \emph{integral} weights, and
  \item\label{item:29} locally \(\fraku_\tau\)--finite.
  \end{enumerate}\label{def:8}
\end{definition}

It follows immediately that \(\catO^\tau\) is a  subcategory of \(\catO(\gl_\enne)\).  Moreover \(\catO^{(1^\enne)} = \catO(\gl_\enne)\). The block decomposition \cref{eq:98} of \(\catO\) induces a block decomposition \(\catO^\tau = \bigoplus_{\boldd \in \weightsgl^+} \catO^\tau_\boldd \). Each block \(\catO^\tau_\boldd\) is a quasi-hereditary category, where the order on the set of weights is the usual dominance order. 

The simple object \(L(\bolda)\) of \(\catO(\gl_\enne)\) is in \(\catO^\tau\) if and only if \(a_i > a_{i+1}\) whenever \(i\) and \(i+1\) are in the same component of the composition \(\tau\), or in other words if and only if \(\bolda \in \weightsgl^+(\gl_\tau)\), where \(\weightsgl^+(\gl_\tau) \subseteq \weightsgl\) denotes the set of integral dominant weights for \(\gl_\tau\). Indeed, \(\catO^\tau\) is the Serre subcategory of \(\catO\) generated by such simple modules. We will denote by \(P^\tau(\bolda)\) the projective cover of \(L(\bolda)\) in \(\catO^\tau\), which is the biggest quotient of \(P^\tau(\bolda)\) which lies in \(\catO^\tau\), that is, \(P^\tau(\bolda)=\sfZ_\tau P(\bolda)\), where \(\sfZ_\tau\colon \catO \mapto \catO^\tau\) is the Zuckerman functor.Suppose that \(\lambda\) is a partition. Then \(L(\bolda) \in \catO^\lambda\) if and only if \(\tableau^\lambda(\bolda)\) is column-strict.

\begin{example}
  Fix \(k=4\) and consider the partition \(\lambda=(3,3,2,1)\). Consider the weight \(\bolda=(3,2,1,4,3,2,2,1,3)\). Then
  \begin{equation*}
    \ytableausetup{centertableaux}
   \tableau^\lambda(\bolda)=  \ytableaushort{1213,232,34}
  \end{equation*}\label{ex:1}
  is column-strict, and indeed \(L(\bolda)\) is a simple module in \(\catO^\lambda\). 
\end{example}

\subsection{The category of presentable modules}
\label{sec:categ-pres-modul}

\begin{definition}
  \label{def:9}
  We define \(\calQ^\tau\) to be the full subcategory of \(\catO^\tau\) consisting of all objects \(M \in \catO^\tau\) which have a presentation
  \begin{equation}
    \label{eq:21}
    P_1 \longrightarrow P_2 \longrightarrow M \longrightarrow 0,
  \end{equation}
  where \(P_1,P_2 \in \catO^\tau\) are \emph{prinjective}, i.e. both,  projective and injective, objects.
\end{definition}

Let \(\boldd \in \weightsgl^+(\frakg)\) be a dominant weight. Then by definition the block \(\calQ^\tau_\boldd\) is equivalent to the category of finite-dimensional modules over the endomorphism algebra \(\End_\frakg(Q)\), where \(Q\) is the sum of the indecomposable prinjective modules of \(\catO^\tau_\boldd\) up to isomorphism. In particular, \(\calQ^\tau\) is a Serre quotient of \(\catO^\tau\) (cf.\ \cref{sec:serre-subc-serre-1}), hence inherits an abelian structure and is a Schurian category. Again, the block decomposition of \(\catO^\tau\) induces a decomposition \(\calQ^\tau = \bigoplus_{\boldd \in \weightsgl^+} \calQ^\tau_\boldd\), which we also call a \emph{block decomposition}.

Let now \(\lambda \in \Part(\enne)\) be a partition of \(\enne\).

\begin{lemma}
  \label{lem:3}
  The indecomposable projective module \(P^\lambda(\bolda)\) of \(\catO^\lambda\) is also injective (i.e. prinjective) if and only if the tableau \(\tableau^\lambda(\bolda)\) is semi-standard.
\end{lemma}

\begin{proof}
  First, consider the case of a regular weight \(\bolda\). Let \(\boldd \in \weightsgl^+(\gl_n)\) be such that \(z \boldd = \bolda\) for \(z \in W\). By \cite[Theorem~5.1]{MR2369489}, the projective module \(P^\lambda(w\boldd)\) is also injective if and only if \(w\) lies in the right cell of the longest element \(w_\lambda\) of  the  shortest left coset representatives \(\shortestcosetleft{\bbS_\lambda}{ \bbS_n}\). Let \(\tableau_i(w)\) and \(\tableau_r(w)\) denote the insertion and the recording tableaux of the Robinson-Schensted correspondence, respectively (see for example \cite[Chapter~4]{MR1464693}). Then \(w\) lies in the right cell of \(w_\lambda\) if and only if \(\tableau_r(w)=\tableau_r(w_\lambda)\), see \cite{MR560412}. One can easily notice that \(\tableau_i(w_\lambda)=\tableau_r(w_\lambda)=\tableau^\lambda(w_\lambda \boldd)\). Moreover, one gets \(\tableau_i(w)=\tableau^\lambda(w\boldd)\) and \(\tableau_r(w)=\tableau^\lambda(w_\lambda \boldd)\) if and only if \(\tableau^\lambda(w\boldd)\) is a standard tableau.\\
Now, if \(\bolda\) is singular, then the claim follows using translation functors, which are known to send prinjective modules to prinjective modules.
In detail, let \(\boldd\) be the dominant weight in the same orbit of \(\bolda\), that is \(\bolda = w \boldd\) with \(w \in W\) of minimal length. Let also \(\bbS_\boldd\) be the stabilizer of \(\boldd\) and let \(w_\boldd\) be its longest element. Let moreover \(\boldr\) be a regular dominant weight. Denote by \(\sfT_\boldr^\boldd\colon \catO_\boldr \mapto \catO_\boldd\) the translation functor. Then \(\sfT_\boldr^\boldd(P^\lambda(ww_\boldd\boldr))\) contains, as a direct summand, \(P^\lambda(\boldd)\), while  \(\sfT_\boldd^\boldr(P^\lambda(\boldd)) = P^\lambda(w w_\boldd \boldr)\) (cf.\ \cite[Chapter~7]{MR2428237}; the results in \emph{loc.\ cit.\ }can be easily adapted to the parabolic category \(\catO^\tau\) by using Zuckerman functor, which commute with translation functors). Now, by construction \(\tableau^\lambda(\boldb)\) is semi-standard if and only if \(\tableau^\lambda(w w_\boldd \boldr)\) is standard, this happens if and only if \(P^\lambda(w w_\boldd \boldr)\) is prinjective, which is the case if and only if \(P^\lambda(\boldb)\) is prinjective.
\end{proof}
It follows that the indecomposable projective modules of \(\calQ^\lambda\) are, up to isomorphism, the \(P^\lambda(\bolda)\)'s for \(\bolda\in \weightsgl\) such that \(\tableau^\lambda(\bolda)\) is semi-standard. We will denote the unique simple quotient of \(P^\lambda(\bolda)\) in \(\calQ^\lambda\) by \(S^\lambda(\bolda)\). We chose the notation \(S^\lambda(\bolda)\) (and not, for example, \(L^\lambda(\bolda)\)) in order to emphasize that \(S^\lambda(\bolda)\) is not, in general, an irreducible \(\frakg\)--module.

\subsection{Standardization functor}
\label{sec:stand-funct}

We fix a composition \(\sigma=(\enne_1,\dotsc,\enne_m)\) of \(\enne\) and abbreviate \(\frakl=\frakl_ \sigma\) and  \(\fraku=\fraku_ \sigma\), and we denote \(\mathfrak{p}_{\mathfrak{l}} = \mathfrak{p}_\sigma\) with corresponding Levi decomposition  \(\frakp_\frakl = \frakl \oplus \fraku_\frakl\). 
We denote by \(W_\frakl=\bbS_{\enne_1} \times \dotsb \times \bbS_{\enne_m} \subset W\) the Weyl group of \(\frakl\) and by \(\weightsgl\) both the integral weights of \(\frakg\) and of \(\frakl\) (since they coincide). Let \(\weightsgl^+(\frakg) \subseteq \weightsgl^+(\frakl)\subset \weightsgl\) be the dominant weights of \(\frakg\) and of \(\frakl\) respectively.

Analogously as we did for \(\gl_\enne\), one can define the category \(\catO(\frakl)\). Note that this category can be identified with the outer tensor product \(\catO(\gl_{\enne_1}) \boxtimes \dotsb \boxtimes \catO(\gl_{\enne_m})\). The purpose of this subsection is to identify this outer tensor product category with a subquotient category of \(\catO(\gl_{\enne})\).

Define the \emph{standardization functor} \(\Delta\colon \catO(\frakl) \mapto \catO(\frakg)\) by parabolic induction:
\begin{equation}
  \label{eq:24}
  \Delta(M) = U(\frakg) \otimes_{U(\frakp_\frakl)} M' = U(\frakg) \otimes_{\frakp_\frakl} M'.
\end{equation}
where $M'=M\otimes \mathbb{C}_\gamma$ is $M$ twisted by the one-dimensional representation $\mathbb{C}_\gamma$ for $\frakh$ (extended by zero to  $\frakp$), of weight $\gamma=\rho(\frakg)-\rho(\frakl)$ (i.e. the difference of the $\rho$ attached to $\frakg$ as in \eqref{eq:62} and the sums of $\rho$'s attached to the factors of $\frakl$ using the analog of formula \eqref{eq:62}).
The following is an immediate standard result:

\begin{lemma}
  \label{lem:13}
  The functor \(\Delta\) is well-defined and exact.
\end{lemma}

\begin{proof}
 Obviously \(\Delta(M)\) is an object of \(\catO(\frakg)\).
By the PBW Theorem
\(U(\frakg)\) is free as a right \(U(\frakp_\frakl)\)--module, hence \(\Delta\) is exact.
\end{proof}

Fix a dominant weight \(\boldd\in \weightsgl^+(\frakg)\) and consider the block \(\catO_\boldd\).  
Let \(\Xi=\Xi_\boldd\) denote the quotient of \(W\boldd\) modulo the action of \(W_\frakl\), and let \(p=p_\boldd \colon W\boldd \surto \Xi\) be the projection. An element \(\xi \in \Xi\) is the lateral class \(W_\frakl \tilde\xi\) of a unique dominant weight \(\tilde \xi \in \weightsgl^+(\frakl)\). The dominance order \(\leq\) on weights restricted to \(\{\tilde \xi \suchthat \xi \in \Xi\}\) gives a partial order \(\preleq\) on \(\Xi\). Via the map \(p\), this induces a preorder \(\preleq\) on \(W\boldd\).
In the following, we will often write \(\catO(\frakl)_\xi\) instead of \(\catO(\frakl)_{\tilde \xi}\).

For the rest of the subsection we fix a \(\xi \in \Xi\). Let \(\catO(\frakg)_{\preleq \xi}\) (respectively, \(\catO(\frakg)_{\prel \xi}\)) be the Serre subcategory of \(\catO(\frakg)_{\boldd}\)
generated by the simple modules \(L(\bolda)\) with \(p(\bolda) \preleq \xi\) (respectively, \(p(\bolda) \prel \xi\)).
For \(\boldb \in P\) we denote also by \(\catO(\frakg)_{\leq \boldb}\) the full subcategory consisting of all modules of \(\catO(\frakg)\) with all weights smaller or equal to \(\boldb\). 

\begin{lemma}
  \label{lem:16}
  Considering  \(\tilde \xi\) as a weight for \(\frakg\), we have
 \(\catO(\frakg)_{\preleq \xi} = \catO(\frakg)_\boldd \cap \catO(\frakg)_{\leq \tilde\xi}\).
\end{lemma}

\begin{proof}
  Since both are Serre subcategories, it is enough to prove that they have the same simple modules. Since the inclusion \(\catO(\frakg)_{\preleq \xi} \subseteq \catO(\frakg)_\boldd \cap \catO(\frakg)_{\leq \tilde\xi}\) is clear, it remains to prove the converse.   Write \(\tilde \xi = z \boldd\) with \(z \in W\) and pick some \(L(w \boldd) \in \catO(\frakg)_{\leq \tilde\xi}\) for \( w \in W\). Suppose that \(w\) and \(z\) are shortest right coset representatives for \(\rightquotient{W}{\bbS_\boldd}\), where \(\bbS_\boldd \) is the stabilizer of \(\boldd\). Notice that since \(z \boldd \in \weightsgl^+(\frakl)\), the element \(z\) is also a shortest left coset representative for \(\leftquotient{W_\frakl}{ W}\). Then \(L(w \boldd) \in \catO(\frakg)_{\leq \tilde\xi}\) implies that \(w \leq z\) in the Bruhat order. To prove that \(L(w \boldd) \in \catO(\frakg)_{\preleq \xi}\) it is enough to show that \(p(w \boldd) \preleq \xi\). Write \(w=xw'\) with \(x \in W_\frakl\) and \(w'\)  a shortest left coset representative for \(\leftquotient{W_\frakl}{ W}\).  Then it is enough to show that \(w' \leq z\) in the Bruhat order. This follows from \cite[Proposition~2.5.1]{MR2133266}.
\end{proof}

Consequently, the functor \(\Delta\)  has image in \(\catO(\frakg)_{\preleq \xi}\) when restricted to \(\catO(\frakl)_\xi\). Let  \(\Delta_{\xi} \colon \catO(\frakl)_{\xi} \mapto \catO(\frakg)_{\preleq \xi}\) be the resulting functor.
Define the functor \(\pi_\xi \colon \catO(\frakg)_{\preleq \xi} \mapto \catO(\frakl)_\xi\) 
\begin{equation}
\label{eq:89}
  \pi_\xi \colon \catO(\frakg)_{\preleq \xi} \into \catO(\frakg) \xrightarrow{\,\calF\,}  \tilde \catO(\frakl) \surto \catO(\frakl)_\xi
\end{equation}
as the composition of inclusion \(\catO(\frakg)_{\preleq \xi} \into \catO(\frakg)\), the functor \(\calF= \Hom_{\frakp_\frakl}(U(\frakl), \res^{\frakg}_{\frakp_\frakl} \blank)\) and of the projection \(\pr_\xi\colon \tilde\catO(\frakl) \mapto \catO(\frakl)_\xi\). Here \(\tilde\catO(\frakl)\) is the full subcategory of \(\frakl\)--modules which satisfy \ref{item:31} and \ref{item:32} and such that all weight spaces are finite dimensional. It is the direct product of the blocks of \(\catO(\frakl)\).

\begin{remark}
  \label{rem:1}
  Note that an element of \(\calF(M)\) is uniquely determined by the image of \(1 \in U(\frakl)\) in \(\res^{\frakg}_{\frakp_\frakl} M\), which can be any vector of
  \begin{equation}
M^{\fraku_\frakl} = \{v  \in  M \suchthat \fraku_\frakl v=0\}.\label{eq:90}
\end{equation}
Since \([\frakl,\fraku_\frakl] \subseteq \fraku_\frakl\), the vector space \(M^{\fraku_\frakl}\) is an \(\frakl\)--subrepresentation of \(M\).  In fact, \(M^{\fraku_\frakl} \cong \calF(M)\) as \(\frakl\)--representations, naturally in \(M\).
Since \(M^{\fraku_\frakl}\) is an \(\frakl\)--submodule of \(M\), it is an \(\frakh\)--weight module and locally \(\frakn_\frakl^+\)--finite. Moreover, its weight spaces are finite dimensional. It follows that \(M^{\fraku_\frakl} \in \tilde\catO(\frakl)\). It is however not clear that \(M^{\fraku_\frakl}\) is finitely generated as a \(U(\frakl)\)--module: this is why we introduced the category \(\tilde\catO(\frakl)\).
\end{remark}

\begin{lemma}
  \label{lem:18}
 The functor \(\pi_\xi\colon \catO(\frakg)_{\preleq \xi} \mapto \catO(\frakl)_\xi\)
  is right adjoint to \(\Delta_\xi\).
\end{lemma}

\begin{proof}
The functor \(\Delta\) is the composition of first extending the \(\frakl\)-action trivially to \(\frakp\) and then apply the induction functor \(U(\frakg)
  \otimes_{U(\frakp_\frakl)} \blank\). On the other hand,  \(\calF\) is the composition of the restriction functor
  \(\res^{\gl_\enne}_{\frakp_\frakl}\) and of the
  co-induction functor \(\Hom_{\frakp_\frakl}(U(\frakl),
  \blank)\), which are the respective adjoint functors (cf.\
  \cite{MR672956}). Hence \((\Delta_\xi, \pi_\xi)\) is an adjoint pair of functors.
\end{proof}

\begin{lemma}
  \label{lem:17}
  We have \(\pi_\xi \circ \Delta_\xi \cong \id\).
\end{lemma}

\begin{proof}
  First, note that if \(N \in \catO(\frakg)_{\preleq \xi}\) then \(\pi_\xi (N)\) can be identified, by \Cref{rem:1}, with an \(\frakl\)--submodule of \(N\). Now let \(M \in \catO(\frakl)_\xi\). By the PBW Theorem, if \(\{m_\beta\}\) is a basis of \(M\) and \(\{x_\gamma\}\) is a  monomial basis of \(\fraku_\frakl^-\),  then a basis of \(U(\frakg) \otimes_{\frakp_\frakl} M\) is given by \(\{x_\gamma \otimes m_\beta\}\). Now, the center \(\frakz_\frakl\) of \(\frakl\) acts on \(\catO(\frakl)_\xi\) according to the (shifted) weight \(\tilde \xi \) restricted to \(\frakz_\frakl\). But unless \(x_\gamma=1\), the action of \(\frakz_\frakl\) on \(x_\gamma \otimes m_\beta\) is given by a (shifted) weight which is strictly smaller than \(\tilde \xi\). Hence \(\pi_\xi (\Delta_\xi (M))\) is an \(\frakl\)--submodule of \(1 \otimes M\).  On the other side, all vectors \(1 \otimes m\) for \(m \in M\) are \(\fraku_\frakl\)--invariant. Hence \(\pi_\xi(\Delta_\xi(M)) = \{1 \otimes m \in U(\frakg) \otimes_{\frakp_\frakl}  M\} \cong M\).
\end{proof}

\begin{lemma}
  \label{lem:19}
  The functor \(\pi_\xi\) is exact.
\end{lemma}

\begin{proof}
  Since \(\pi_\xi\) is a right-adjoint functor, it is left exact. So we only need to prove that \(\pi_\xi\) sends epimorphisms to epimorphisms. Let \(f \colon M \surto N\) be an epimorphism in \(\catO(\frakg)_{\preleq \xi}\), and consider the restriction \(\calF(f)=f_{| M^{\fraku_\frakl}}\colon M^{\fraku_\frakl} \mapto N^{\fraku_\frakl}\). Let \(v \in \pr_\xi (N^{\fraku_\frakl})\) and choose a preimage \(v' \in M\) such that \(f(v')=v\). Suppose, without loss of generality, that all our vectors are weight vectors.
Suppose that there exists a \(u \in \fraku_\frakl\) such that \(u v' \neq 0\). Then the weight of \(u v'\) is not a weight of any object of \(\catO(\frakg)_{\preleq \xi}\), and this cannot happen. Hence \(\fraku_\frakl v'=0\) and \(\pi_\xi(f)\) is surjective.
\end{proof}

\begin{prop}
  \label{prop:21}
  The pair \((\catO(\frakl)_\xi, \pi_\xi)\) is the quotient category \(\catO(\frakg)_{\preleq \xi}/\catO(\frakg)_{\prel \xi}\).
\end{prop}

\begin{proof}
  We check that \((\catO(\frakl)_\xi, \pi_\xi )\) satisfies the universal property of the Serre quotient category. By \Cref{lem:19}, \(\pi_\xi\) is exact and, by definition, \(\pi_\xi\)  vanishes on \(\catO(\frakg)_{\prel \xi}\) (in particular, the last functor in the composition \cref{eq:89} defining \(\pi_\xi\) kills \(\catO(\frakg)_{\prel \xi}\)). Let now \(\calA\) be any abelian category and \(\calG\colon \catO(\frakg)_{\preleq \xi} \mapto \calA\) an exact functor which vanishes on \(\catO(\frakg)_{\prel \xi}\). Define \(\bar \calG\colon \catO(\frakl)_\xi \mapto \calA\) to be \(\bar \calG = \calG \circ \Delta_\xi\). We shall prove that \(\bar \calG \circ \pi_\xi= \calG\).

Consider the adjunction morphism \(\psi\colon \Delta_\xi \circ \pi_\xi \mapto \id\). For \(M \in \catO(\frakg)_{\preleq \xi}\) the map \(\psi_M \colon U(\gl_\enne) \otimes_{\frakp_\frakl} \pr_\xi(M^{\fraku_\frakl}) \mapto M\) is simply \(\psi_M(x \otimes v)= xv\). Let \(w = \sum_i x_i \otimes v_i\) be some weight vector in \(\ker(\psi_M)\), i.e. \(\psi_M(w)=0\). Using the PBW Theorem, write each \(x_i\) as \(x_i=x_i'x_i''\) with \(x_i' \in U(\fraku^-_\frakl)\) and \(x_i'' \in U(\frakp_\frakl)\). Then \(x_i \otimes v_i = x_i' \otimes x_i'' v_i\). Since \(w\) is a weight vector, we can suppose that all \(x_i\)'s are either \(1\) or in \(\fraku^-_\frakl U(\fraku^-_\frakl)\). In the first case we have \(w=0\), while in the second case the weight of \(w\) is not \(W_\frakl\)--linked to \(\tilde\xi\). This shows that all weights of \(\ker(\psi_M)\) are strictly smaller than \(\tilde\xi\) and not \(W_\frakl\)--linked to \(\tilde\xi\), and hence \(\ker (\psi_M)\) is an object of \(\catO(\frakg)_{\prel \xi}\). \\
Let us now consider the cokernel of \(\psi_M\). 
Let \(\{v_\alpha\suchthat \alpha \in A\}\) be generators of \(M\) as an \(\frakl\)--module.
Of course they generate \(M\) as a \(\frakg\)--module, and their images \(\{\bar v_\alpha\}\) in the quotient generate \(\coker \psi_M\). Suppose that \(v_\alpha\) has weight \(W_\frakl\)--linked to \(\tilde\xi\). Then \(\fraku_\frakl v_\alpha=0\)  by weight considerations. Hence \(v_\alpha = \psi(1 \otimes v_\alpha)\) and \(\bar v_\alpha=0\). Therefore \(\coker \psi_M\) is an object of \(\catO(\frakg)_{\prel \xi}\), too.\\
It follows that for each \(M \in \catO(\frakg)_{\preleq \xi}\) we have an exact sequence
\begin{equation}
  \label{eq:106}
  0 \mapto \ker \psi_M \mapto \Delta_\xi \circ \pi_\xi (M) \mapto M \mapto \coker \psi_M \mapto 0
\end{equation}
with both \(\ker \psi_M\) and \(\coker \psi_M\) in \(\catO(\frakg)_{\prel \xi}\). 
 Since \(\calG\) is exact and vanishes on \(\catO(\frakg)_{\prel \xi}\), applying \(\calG\) to \cref{eq:106} implies \(\bar \calG\circ \pi_\xi \cong \calG\) as claimed. (Note that this proves at once that \(\bar \calG\) is uniquely determined up to isomorphism.)
\end{proof}

Let \(\frakq \subseteq \frakl\) be a standard parabolic subalgebra and let \(\hat \frakq = \frakq + \frakb \subseteq \frakg\). Then the same results we just proved hold for the parabolic categories \(\catO(\frakg)^{\hat \frakq}\) and \(\catO(\frakl)^\frakq\). In particular, we have a pair \((\Delta_\xi, \pi_\xi)\) of adjoint functors
\begin{equation}
  \label{eq:91}
  \begin{tikzpicture}[anchorzero]
    \node (A) at (0,0) {$\catO(\frakg)^{\hat \frakq}_{\preleq \xi}$};
    \node (B) at (3,0) {$\catO(\frakl)^\frakq_\xi$.};
    \draw[->,transform canvas={yshift=1mm}] (A) to node[above] {$\pi_\xi$} (B);
    \draw[->,transform canvas={yshift=-1mm}] (B) to node[below] {$\Delta_\xi$} (A);
  \end{tikzpicture}
\end{equation}

\subsection{Generalized parabolic subcategories}
\label{sec:gener-parab-subc}

Let now \(\boldlambda=(\lambda^{(1)},\dotsc,\lambda^{(m)}) \in \Part^m(\enne)\) be a multipartition of \(\enne\) with \(m\) parts. Let \(\enne_i = \abs{\lambda^{(i)}}\) for \(i=1,\dotsc,m\), and let as before \(\frakl= \gl_{\enne_1} \oplus \dotsb \oplus \gl_{\enne_m}\).
The outer tensor product \(\calC^\boldlambda = \calQ^{\lambda^{(1)}} \boxtimes \dotsb \boxtimes \calQ^{\lambda^{(m)}}\) (see \cref{sec:outer-tensor-product}) can be considered as a full subcategory of \(\catO(\frakl)\). 

\begin{lemma}
  \label{lem:2}
  The category \(\calC^\boldlambda\) 
 satisfies the following properties:
  \begin{enumerate}[(i)]
  \item \label{item:1} it is stable under tensor product with finite-dimensional \(\frakl\)--modules;
  \item \label{item:2} it decomposes into a direct sum of full subcategories, each equivalent to a module category over a finite-dimensional self-injective associative algebra;
  \item \label{item:3} the action of the center of \(\frakl\) on any object \(M \in \calC^\boldlambda\) is diagonalizable.
  \end{enumerate}
\end{lemma}

\begin{proof}
  The properties hold for \(\calC^\boldlambda\) since they hold for each tensor factor.
\end{proof}

Using the terminology of \cite{MR2450613}, the lemma implies that the category \(\calC^\boldlambda\) is admissible. 
Following  \cite{MR1921761}, we can define generalized parabolic subcategories:

\begin{definition}
  \label{def:10}
  We let \(\calQ^\boldlambda\) be the full subcategory of all \(\gl_\enne\)--modules which are
  \begin{enumerate}[(Q1)]
  \item \label{item:6} finitely generated as \(U(\gl_\enne)\)--modules,
  \item \label{item:7} as \(\frakl\)--modules, a direct sum of objects of \(\calC^\boldlambda\),
  \item \label{item:8} locally \(\fraku_\frakl\)--finite.
  \end{enumerate}
\end{definition}

The three conditions imply immediately that \(\calQ^\boldlambda\) is a full subcategory of \(\catO(\gl_\enne)\). Hence it inherits from the block decomposition of \(\catO(\gl_\enne)\) a decomposition, which we still call a \emph{block decomposition}.

\begin{lemma}
  \label{lem:20}
  The standardization functor  restricted to \(\calC^\boldlambda\) has values in \(\calQ^\boldlambda\).
\end{lemma}

\begin{proof}
  Let \(M \in \catO(\frakl)\). It follows by the PBW Theorem that \(\Delta(M)\) decomposes, as an \(\frakl\)--module, into a direct sum of modules isomorphic to \(M\). In particular, if \(M \in \calC^\boldlambda\) then \ref{item:7} holds for \(\Delta(M)\).
\end{proof}

We fix \(\tau\) to be the composition
  \begin{equation}
    \label{eq:23} \tau = (\lambda^{(1)}_1,\lambda^{(1)}_2,\dotsc,\lambda^{(2)}_1,\lambda^{(2)}_2,\dotsc).
  \end{equation}
We  denote  \(\catO(\frakl)^\tau\) the parabolic subcategory of \(\catO(\frakl)\) corresponding to the partition \(\tau\), that can be also identified with \( \catO^{\lambda^{(1)}} \boxtimes \dotsb \boxtimes \catO^{\lambda^{(m)}}\). By definition, \(\calC^\boldlambda\) is a Serre quotient of \(\catO(\frakl)^\tau\). Analogously, we claim that \(\calQ^\boldlambda\) is a Serre quotient of \(\catO(\frakg)^\tau\), i.e. a ``subquotient category'' of \(\catO(\frakg)\):

\begin{prop}
  \label{prop:4}
 The category \(\calQ^\boldlambda\) coincides with
the category of \(\calP^\boldlambda\)--presentable modules in \(\catO(\frakg)^\tau\), where \(\calP^\boldlambda\) is the additive category generated by the projective modules \(P^{\tau}(\bolda)\) for \(\bolda \in \weightsgl\) such that \(\tableau^{\boldlambda}(\bolda)\) is semi-standard.
\end{prop}

\begin{proof}
 Let us denote by \((\calP^\boldlambda)^\pres\) the category of \(\calP^\boldlambda\)--presentable modules. We start with the inclusion \((\calP^\boldlambda)^\pres \subseteq \calQ^\boldlambda\). Let \(\bolda \in \weightsgl\) be a weight such that \(\tableau^\boldlambda(\bolda)\) is semi-standard. Choose \(\boldb\) maximal in the \(W\)--orbit of \(\bolda\) such that \(\tableau^\boldlambda(\boldb)\) is also semi-standard. Let \(P^\tau_\frakl(\boldb) \in \catO(\frakl)^\tau\) be the projective cover of the simple module \(L_\frakl(\boldb) \in \catO(\frakl)^\tau\). Then \(\Delta(P^\tau_\frakl(\boldb)) \in \catO(\frakg)^\tau\) is projective (since \(\boldb\) is in the same \(W_\frakl\)--orbit of a \(\frakg\)--dominant weight), and \(\Delta^\tau(\boldb)\) is the parabolic Verma module with minimal weight appearing in a Verma filtration of it. Hence, being indecomposable, we have \(\Delta(P^\tau_\frakl(\boldb))=P^\tau(\boldb)\). 
 Since \(P^\tau_\frakl(\boldb) \in \calC^\boldlambda\), by \Cref{lem:20} we have that \(P^\tau(\boldb) \in \calQ^\boldlambda\). By tensoring \(P^\tau(\boldb)\) with finite-dimensional modules we can generate  \(P^\tau(\bolda)\) as a direct summand, and since \(\calQ^\boldlambda\) is closed under tensoring with finite-dimensional modules, we have \(P^\tau(\bolda) \in \calQ^\boldlambda\). Now, if \(M \in (\calP^\boldlambda)^\pres\), let \(P \mapto Q \surto M\) be a \(\calP^\boldlambda\)--presentation. Consider this as a sequence of \(\frakl\)--modules: \(P\) and \(Q\) decompose into a direct sum of objects from \(\calC^\boldlambda\) and so \(M\) decomposes into a direct sum of \(\frakl\)--modules which have a presentation via objects from \(\calC^\boldlambda\), hence \(M \in \calQ^\boldlambda\).\\
For the converse, let \(M \in \calQ^\boldlambda\). It follows immediately from the property \ref{item:7} that \(M \in \catO(\frakg)^\tau\). As an \(\frakl\)--module, \(M\) is generated by vectors of (shifted) weight \(\bolda\) such that \(\tableau^\boldlambda(\bolda)\) is semi-standard. Of course, this is also true as \(\frakg\)--module. Hence the projective cover \(Q\) of \(M\) in \(\catO(\frakg)^\tau\) is an element of \(\calP^\boldlambda\). Let \(K = \ker(Q \surto M)\) in \(\catO(\frakg)^\tau\). Consider the exact sequence \(K \into Q \surto M\) of \(\frakl\)--modules. Up to taking direct summands, we may suppose that this is a sequence of finitely generated \(U(\frakl)\)--modules, with \(M\in \calC^\boldlambda\) and, by the other inclusion proved in the previous paragraph, also \(Q \in \calC^\boldlambda\). Write \(Q=Q_M \oplus Q'\), where \(Q_M\) is the projective cover of \(M\), and \(K= Q' \oplus \ker (Q_M \surto M)\). Since \(M \in \calC^\boldlambda\), we have a presentation \(P_M \mapto Q_M \surto M\) with \(P_M\) prinjective in \(\catO(\frakl)^\tau\), hence we have a surjective map \(P_M \surto \ker(Q_M \surto M)\) and therefore a surjective map \(P' \surto K\) with \(P' = Q' \oplus P_M\). Notice that \(P'\) is also a prinjective object of \(\catO(\frakl)^\tau\). Hence it is generated by vectors of (shifted) weight \(\bolda\) such that \(\tableau^\boldlambda(\bolda)\) is semi-standard. The same holds obviously for  \(K\). Hence its projective cover \(P\) is in  \(\calP^\boldlambda\) and \(M\) has a presentation \(P \mapto Q \surto M\) with \(P,Q \in \calP^\boldlambda\). Hence  \((\calP^\boldlambda)^\pres \supseteq \calQ^\boldlambda\).
\end{proof}

\begin{remark}
  \label{rem:5}
  \Cref{prop:4} and its proof generalize \cite[Proposition~5.3.2]{miophd2}.
\end{remark}

Let \(\boldd\in \weightsgl^+(\frakg)\) be a dominant weight and \(P_\boldd\) a generator of \(\calP^\boldlambda_\boldd\). Then we have an equivalence of categories \(\calQ^\boldlambda_\boldd \cong \rmod{\End_\frakg(P_\boldd)}\) and obtain immediately:

\begin{corollary}
  \label{cor:2}
  The category \(\calQ^\boldlambda\), identified with the category of \(\calP^\boldlambda\)--presentable modules in \(\catO^\tau(\frakg)\), is a Schurian category.
\end{corollary}

Hence the indecomposable projective modules of \(\calQ^\boldlambda\) are the \(P^\tau(\bolda)\) for \(\bolda \in \weightsgl\) such that \(\tableau^\boldlambda(\bolda)\) is semi-standard. We will write \(P^\boldlambda(\bolda)\) for \(P^\tau(\bolda)\) when we consider it as an object of \(\calQ^\boldlambda\).  Let \(S^\boldlambda(\bolda)\) be the unique simple quotient of \(P^\boldlambda(\bolda)\) in \(\calQ^\boldlambda\). (Note that \(S^\boldlambda(\bolda)\) is not, in general, an irreducible \(\frakg\)--module!) Then the \(S^\boldlambda(\bolda)\)'s for \(\bolda \in \weightsgl\) such that \(\tableau^\boldlambda(\bolda)\) is semi-standard give the set of simple objects of \(\calQ^\boldlambda\) up to isomorphism.

\begin{prop}\label{prop:1}
 Parabolic induction defines an exact  standardization functor 
 \begin{equation}
   \Delta \colon \calC^\boldlambda \longrightarrow \calQ^\boldlambda.
\end{equation}
\end{prop}

\begin{proof}
 Note that the claim is non-trivial since the abelian structure of the Serre quotient category $\calC^\boldlambda$ is not the same as the abelian structure of the whole category \(\catO\).  Let therefore  \(M_1 \into M_2 \surto M_3\) be a short exact sequence in \(\calC^\boldlambda\). Since the inclusion \(\sfi \colon \calC^\boldlambda \rightarrow \catO(\frakl)\) is right exact, we have an exact sequence
\begin{equation}
  \label{eq:9}
  0 \longrightarrow K \longrightarrow \sfi M_1 \longrightarrow \sfi M_2 \longrightarrow \sfi M_3 \longrightarrow 0
\end{equation}
in \(\catO(\frakl)\), where \(K\) is just the kernel of the map \(\sfi M_1 \rightarrow \sfi M_2\) in \(\catO(\frakl)\). Actually the exact sequence \eqref{eq:9} lives in the parabolic category \(\catO(\frakl)^\tau\), and, by construction, \(\mathsf C_\frakl (K) = 0 \), where \(\mathsf C_\frakl \colon \catO(\frakl)^\tau \mapto \calC^\boldlambda\) is the coapproximation functor. We apply \(\Delta\) to \eqref{eq:9} and obtain the exact sequence 
\begin{equation}
  \label{eq:10}
    0 \longrightarrow \Delta K \longrightarrow \Delta \sfi M_1 \longrightarrow \Delta \sfi M_2 \longrightarrow \Delta \sfi M_3 \longrightarrow 0
\end{equation}
in \(\catO(\frakg)\), or actually in \(\catO(\frakg)^\tau\). In order to conclude that \(\Delta \sfi M_1 \into \Delta \sfi M_2 \surto \Delta \sfi M_3\) is exact in \(\calQ^\boldlambda\), we need to show that \(\mathsf C_\frakg (\Delta K) = 0\), where here \(\mathsf C_\frakg\) denotes the coapproximation functor \(\catO(\frakg)^\tau \mapto \calQ^\boldlambda\) provided by Proposition~\ref{prop:4}. Since \(\Delta\) and \(\mathsf C_\frakg\) are exact, it is enough to show the claim for \(K=L_\frakl(\bolda)\) a simple module in \(\catO(\frakl)^\tau\).

Let therefore \(L_\frakl(\bolda) \in \catO(\frakl)^\tau\) be a simple module such that \(\mathsf C_\frakl (L_\frakl(\bolda)) = 0\).  Now, it follows from the PBW Theorem and from the fact that \(U(\frakg)\) is locally finite for the adjoint action, that \(\Delta L_\frakl(\bolda)\), as an \(\frakl\)--module, decomposes into a direct sum of modules which are obtained from \(L_\frakl(\bolda)\) by tensoring with finite-dimensional \(\frakl\)--modules.
Let \(\boldb\) be any weight such that \(\tableau^\boldlambda(\boldb)\) is semi-standard, and let \(P^\boldlambda_\frakl(\boldb) \in \catO(\frakl)^\tau\) be the corresponding projective module (which by assumption lies in \(\calA^\boldlambda\) and is also injective). Let \(E\) be a finite dimensional \(\frakl\)--module. Then we have
\begin{equation}
  \label{eq:11}
  \Hom_{\catO(\frakl)^\tau}\big(P_\frakl^\boldlambda(\boldb),E \otimes  L_\frakl(\bolda) \big) \cong \Hom_{\catO(\frakl)^\tau}\big(E^* \otimes P_\frakl^\boldlambda(\boldb), L_\frakl(\bolda) \big) = 0,
\end{equation}
where the last equality follows since \(E^* \otimes P_\frakl^\boldlambda(\boldb)\) is the direct sum of projective modules which by Lemma~\ref{lem:2}~\ref{item:1} lie in \(\calC^\boldlambda\), while by assumption  \(\sfC_\frakl( L_\frakl(\bolda))=0\).
 Hence \(\sfC_\frakl (E \otimes L_\frakl(\bolda))=0\), or in other words all composition factors of \(E \otimes L_\frakl(\bolda)\) are of type \(L_\frakl(\boldc)\), where \(\tableau^\boldlambda (\boldc)\) is column-strict but not semi-standard. By our discussion, the same holds, as an \(\frakl\)--module, for \(\Delta (L_\frakl(\bolda))\), which is just an infinite direct sum of such objects.
 A fortiori it must be true that all composition factors of \(\Delta(L_\frakl(\bolda))\) as a \(\frakg\)--module are of type \(L(\boldc)\), where \(\tableau^\boldlambda (\boldc)\) is column-strict but not semi-standard.  This implies immediately that \(\mathsf C_\frakg \Delta (L_\frakl(\bolda))=0\), and we are done.
\end{proof}

\subsection{Standardly stratified structure}
\label{sec:struct-stand-strat}

As before, we fix a multipartition \(\boldlambda\) of \(\enne\). 

Let \(\boldd\in \weightsgl^+(\frakg)\) be a dominant weight, and fix a block \(\calQ^\boldlambda_\boldd\). 
 Consider \(\St^\boldlambda(\boldd)\), the set of semi-standard multitableaux of shape \(\boldlambda\) and type \(\boldd\). We can view \(\St^\boldlambda(\boldd) \subseteq W\boldd\). In particular, the map \(p\colon W\boldd \mapto \Xi\) from \cref{sec:stand-funct} restricts to a map \(p\colon \St^\boldlambda(\boldd)\surto \Xi\) and induces a preorder \(\preleq \) also on \(\St^\boldlambda(\boldd)\). As we did for \(\catO(\frakl)\), we set \(\calC^\boldlambda_\xi=\calC^\boldlambda_{\tilde \xi}\). 

For \(\xi \in \Xi\) let \(\calQ^\boldlambda_{\preleq \xi}\) (respectively, \(\calQ^\boldlambda_{\prel \xi})\) be the Serre subcategory of \(\calQ^\boldlambda\) generated by the simple objects \(S^\boldlambda(\bolda)\) for \(p(\bolda) \preleq \xi\) (respectively, \(p(\bolda) \prel \xi\)). Let also  \(\calQ^\boldlambda_{\sim \xi}\) be the Serre quotient \(\calQ^\boldlambda_{\preleq \xi} / \calQ^\boldlambda_{\prel \xi}\). 

Note that the category \(\calQ^\boldlambda_{\preleq \xi}\) is the full category of all modules of \(\calQ^\boldlambda_\boldd\) which have weights smaller than or equal to \(\tilde\xi\), that is \(\calQ^\boldlambda_{\preleq \xi} = \calQ^\boldlambda_\boldd \cap \catO(\frakg)_{\leq \tilde \xi}\),  cf. \Cref{lem:16}. (If \(S^{\xi}(\bolda)\) denotes the simple module of \(\calC^\boldlambda_\xi\) corresponding to \(\bolda\) then \(S^\boldlambda(\bolda)\) is a quotient of \(\Delta(S^{\xi}(\bolda))\), hence an object of \(\catO(\frakg)_{\leq \xi}\).) 

Let \(\bolda\) be a weight such that \(\tableau^\boldlambda(\bolda)\) is semi-standard, and let \(\xi=p(\bolda)\). We denote by \(P_\frakl^\xi(\bolda)\) the indecomposable projective object corresponding to \(\bolda\) in \(\calC^\boldlambda_\xi\), and we let \(\Delta^\boldlambda(\bolda)=\Delta_\xi(P^\xi(\bolda))\). As well known, \(P_\frakl^\xi(\bolda)\) is generated by one element \(v\) of (shifted) weight \(\bolda\). Then \(\Delta^\boldlambda(\bolda)\) is generated by \(1 \otimes v\), which has also (shifted) weight \(\bolda\). Therefore there is an epimorphism \(P^\boldlambda(\bolda) \surto \Delta^\boldlambda(\bolda)\). 

\begin{prop}[{\cite{MR1921761}}]
  \label{prop:17}
  The kernel of this epimorphism \(P^\boldlambda(\bolda) \surto \Delta^\boldlambda(\bolda)\) admits a filtration by  standard modules \(\Delta(\boldb)\) with \(p(\boldb)\preg p(\bolda)\).
\end{prop}

The proof is based on the fact that for \(h\) large enough the module
\begin{equation}
U(\frakg) \otimes_\frakp \big(U(\fraku_\frakl)/U(\fraku_\frakl)\fraku_\frakl^h \otimes P^\xi(\bolda)\big)\label{eq:94}
\end{equation}
is projective in \(\calQ^\boldlambda\), contains \(P^\boldlambda(\bolda)\) as a direct summand  and has the required standard filtration.
For more details see \cite[Theorem~3]{MR1921761}. It is also possible to prove the result mimicking the corresponding result for the BGG category \(\catO\), see \cite[Chapter~3]{MR2428237}, although the combinatorics gets quite tricky.

\begin{lemma}
  The functor \(\pi_\xi\), restricted to
  \(\calQ^\boldlambda_{\preleq \xi}\), has image in
  \(\calC^\boldlambda_\xi\) and is right adjoint to
  \(\Delta_\xi\colon \calC^\boldlambda_\xi \mapto
  \calQ^\boldlambda_{\preleq \xi}\).\label{lem:15}
\end{lemma}

By a slight abuse of notation, we will denote by \(\pi_\xi\) both the functor defined on \(\catO(\frakg)^\tau_{\preleq \xi}\) and its restriction to \(\calQ^\boldlambda_{\preleq \xi}\), but we will always specify which functor we will be considering.

\begin{proof}
  Let \(\tau\) be the composition \cref{eq:23}. The proof is based on the exactness of the functor \(\pi_\xi \colon \catO(\frakg)^{\tau}_{\preleq \xi} \mapto \catO^{\tau}(\frakl)_\xi\). Since \(\calQ^\boldlambda_{\preleq \xi}\), although being a full subcategory of \(\catO(\frakg)^{\tau}_{\preleq \xi}\), does not inherit the abelian structure of the latter (and the same for \(\calC^\boldlambda_\xi\)), we have to be careful.\\
 Let \(M \in \calQ^\boldlambda_{\preleq \xi}\). Then \(M\) has a presentation \(P \mapto Q \surto M\) with \(P,Q \in \calP^\boldlambda\). In general, \(P\) and \(Q\) will not be objects of \(\catO(\frakg)^\tau_{\preleq \xi}\). We apply the right-exact Zuckerman functor \(\sfZ_{\preleq \xi}\colon \catO(\frakg)^\tau \mapto \catO(\frakg)^\tau_{\preleq \xi}\) and get a presentation \(\sfZ_{\preleq \xi} P \mapto \sfZ_{\preleq \xi} Q \surto  M\). Now, it follows from \Cref{prop:17} above that \(\sfZ_{\preleq \xi}P\) and \(\sfZ_{\preleq \xi}Q\) are filtered, as \(\frakg\)--modules, by standard modules \(\Delta^\boldlambda(\bolda)\).
Since \(\pi_\xi \circ \Delta_\xi \cong \id\) and \(\pi_\xi\) is exact, it follows that \(\pi_\xi(\sfZ_{\preleq \xi} P), \pi_\xi(\sfZ_{\preleq \xi} Q) \in \calC^\boldlambda\), and they are projective. Again, since \(\pi_\xi\) is exact, we have a presentation \(\pi_\xi(\sfZ_{\preleq \xi} P) \mapto \pi_\xi(\sfZ_{\preleq \xi}Q) \surto \pi_\xi(M)\), hence \(\pi_\xi (M) \in \calC^\boldlambda_\xi\).
\end{proof}

\begin{prop}
  \label{prop:6}
  The pair \((\calC^\boldlambda_\xi,\pi_\xi)\) is the Serre quotient of \(\calQ^\boldlambda_{\preleq \xi}\) modulo \(\calQ^\boldlambda_{\prel \xi}\).
\end{prop}
\begin{proof}
We check that \((\calC^\boldlambda_\xi,\pi_\xi)\) satisfies the universal property of the Serre quotient category. The proof is analogous to the proof of \Cref{prop:21}, but again, we need to be a bit more careful since the abelian structure of our categories is not induced by the abelian structure on the category of \(\frakg\)--modules.  First, we observe that the functor \(\pi_\xi\colon \calQ^\boldlambda_{\preleq \xi} \mapto \calC^\boldlambda_\xi\) is exact. Indeed, since it is a right adjoint it is automatically left exact. On the other side, \(\pi_\xi\) is the composition of the following three functors: (i) the inclusion of \(\calQ^\boldlambda_{\preleq \xi}\) into \(\catO(\frakg)^\tau_{\preleq \xi}\), which is right exact, (ii) the functor \(\pi_\xi \colon \catO(\frakg)^\tau_{\preleq \xi}\mapto \catO(\frakl)^\tau_\xi\), which is exact, and (iii) the coapproximation functor \(\catO(\frakl)^\tau_\xi \mapto \calC^\boldlambda_\xi\), which is also exact. Hence \(\pi_\xi\) is right exact, too.\\
We prove now that \((\calC^\boldlambda_\xi, \pi_\xi)\) satisfies the universal property of the Serre quotient. Let \(\ocalC\) be any abelian category and let \(\calG\colon \calQ^\boldlambda_{\preleq \xi} \mapto \ocalC\) be an exact functor which vanishes on \(\calQ^{\boldlambda}_{\prel \xi}\). Define \(\bar \calG\colon \calC^\boldlambda_\xi \mapto \ocalC\) as \(\bar \calG= \calG \circ \Delta_\xi\). We need to show that \(\bar \calG \circ \pi_\xi \cong \calG\).\\
Consider the adjunction morphism \(\psi\colon \Delta_\xi \circ \pi_\xi \mapto \id\).
For each \(M \in \calQ^\boldlambda_{\preleq \xi}\) we have from \cref{eq:106} an exact sequence of \(\frakg\)--modules
\begin{equation}
  \label{eq:43}
  0 \mapto \ker \psi_M \mapto \Delta \circ \pi_\xi (M) \mapto M \mapto \coker \psi_M \mapto 0
\end{equation}
with both \(\ker \psi_M\) and \(\coker \psi_M\) in \(\catO(\frakg)_{\prel \xi}\). In general, we cannot say that \(\ker \psi_M\) and \(\coker \psi_M\) are objects of \(\calQ^\boldlambda\), but can apply the coapproximation functor \(\sfC\colon \catO(\frakg)^\tau \mapto \calQ^\boldlambda\) to \cref{eq:43} to obtain
\begin{equation}
  \label{eq:92}
  0 \mapto \sfC(\ker \psi_M) \mapto \Delta \circ \pi_\xi (X) \mapto X \mapto \sfC(\coker \psi_M) \mapto 0
\end{equation}
By \Cref{lem:5} below, \(\sfC(\ker \psi_M), \sfC(\coker \psi_M) \in \calQ^\boldlambda_{\prel \xi}\). Since \(\calG\) is exact and vanishes on \(\calQ^\boldlambda_{\prel \xi}\) we have \(\bar \calG\circ \pi_\xi \cong \calG\), and we are done.
\end{proof}

\begin{lemma}
  \label{lem:5}
  The functor \(\sfC\colon \catO(\frakg)^\tau\mapto \calQ^\boldlambda\) restricts to a functor \(\catO(\frakg)^\tau_{\prel \xi} \mapto \calQ^\boldlambda_{\prel\xi}\).
\end{lemma}

\begin{proof}
  We have to show that \(\sfC\) sends \(\catO(\frakg)^\tau_{\prel \xi} \) to \(\calQ^\boldlambda_{\prel\xi}\).
  Since \(\sfC\) is exact, it suffices to prove the claim for simple modules. This is however obvious, since 
  \begin{equation}
    \label{eq:93}
    \sfC(L(\bolda)) \cong
    \begin{cases}
      S^\boldlambda(\bolda) & \text{if \(\tableau^\boldlambda(\bolda)\) is semi-standard},\\
      0 & \text{otherwise.}
    \end{cases}
  \end{equation}\renewcommand{\qedsymbol}{}
\end{proof}

Altogether we obtain:
\begin{theorem}
  \label{thm:3}
  Let \(\boldlambda\) be a multipartition. Each block \(\calQ^\boldlambda_\boldd\) of the category \(\calQ^\boldlambda\) is  a standardly stratified category.
\end{theorem}

\section{Categorical \texorpdfstring{\(\mathfrak{sl}_k\)--}{sl(k)-}action}
\label{sec:acti-degen-affine}

We recall in this section the categorical \(\fraksl_k\)--action on \(\catO(\gl_\enne)\).

\subsection{The degenerate affine Hecke algebra}
\label{sec:degen-affine-hecke}

First, we recall the definition of the degenerate affine Hecke algebra.

\begin{definition}[\cite{MR831053}]
  \label{def:11}
  The \emph{degenerate affine Hecke algebra} \(\ucalH^\aff_r\) is the unital \(\C\)--algebra on generators \(x_1,\dotsc,x_r\) and \(t_1,\dotsc,t_{r-1}\) subject to the following relations:
  \begin{enumerate}[(a)]
  \item \label{item:9}  \(x_i \mapsto x_i\) defines an inclusion of the polynomial ring \(\C[x_1,\dotsc,x_n]\) into \(\ucalH^\aff_r\);
  \item \label{item:10} \(s_i \mapsto t_i\) defines an inclusion of the  group algebra \(\C[\bbS_n]\) of the symmetric group into \(\ucalH^\aff_r\);
  \item \label{item:11} finally, the following commutation relations hold:
    \begin{eqnarray}
      &t_j x_i -x_i t_j = 0, \qquad \text{if } \abs{i-j}>1, &\label{eq:52}\\
     & t_j x_j - x_{j+1} t_j  = 1, \qquad t_j x_{j+1} - x_j t_j = -1.&\label{eq:51}
    \end{eqnarray}
  \end{enumerate}
\end{definition}

Let \(\catO=\catO(\gl_\enne)\) and \(M \in \catO(\gl_\enne)\). There is a well-known action of  \(\ucalH^\aff_r\) on \(M \otimes (\C^\enne)^{\otimes r}\), which we recall briefly.
Let \(X_{bc} \in \gl_\enne\) for \(b,c=1,\dotsc,n\) be the matrix units. Let
\begin{equation}
  \label{eq:6}
  \Omega = \sum_{b,c=1}^\enne X_{bc} \otimes X_{cb}\quad\in \quad U(\gl_\enne) \otimes U(\gl_\enne),
\end{equation}
be the {\em Casimir operator} 
and \(C=m(\Omega)\) the {\em Casimir element} of $U(\gl_\enne)$, where
$m\colon U(\gl_\enne) \otimes U(\gl_\enne) \to U(\gl_\enne) $ is the multiplication.
Define for $0 \leq h < l \leq r$
\begin{equation}
  \label{eq:17}
  \Omega_{hl} = \sum_{b,c=1}^\enne 1 \otimes \cdots \otimes 1 \otimes X_{bc} \otimes 1 \otimes \cdots \otimes 1 \otimes X_{cb} \otimes 1 \otimes \cdots \otimes 1,
\end{equation}
where $X_{bc}$ resp. $X_{cb}$ are the \(h\)--th and $l$--th tensor factor, starting with position $0$. 

Let \(\sigma\colon \C^\enne \otimes \C^\enne \mapto \C^\enne \otimes \C^\enne\) be the map \(v \otimes w \mapsto w \otimes v\). Then we obtain the following

\begin{prop}[\cite{MR1652134}]
\label{prop:10}
  For any \(M \in \catO\), the assignments
  \begin{equation}
    \label{eq:53}
    t_h \mapsto \id \otimes \id^{\otimes(h-1)} \otimes
    \sigma \otimes
    \id^{\otimes(r-h-1)}, \qquad
    x_h  \mapsto \textstyle\sum_{0 \leq l < h} \Omega_{lh}
  \end{equation}
  define an algebra homomorphism \(\Psi_{M,r} \colon \ucalH^\aff_r
  \mapto \End_\catO(M \otimes (\C^\enne)^{\otimes r})\). This map is natural in \(M\), i.e.\ if \(M' \in \catO\) and \(f \in \Hom_\catO(M,M')\) then 
  \begin{equation}
\Psi_{M',r}(z) \circ (f \otimes \id^{\otimes r}) = (f \otimes \id^{\otimes r}) \circ\Psi_{M,r}(z)\label{eq:55}
\end{equation}
 for all \(z \in \ucalH^\aff_r\).  \label{prop:7}
\end{prop}

In particular, if 
\begin{equation}
  \label{eq:57}
    \sfF \colon \catO \mapto \catO, \quad
    M  \mapsto M \otimes \C^\enne
\end{equation}
denotes the standard translation functor, then we have:

\begin{corollary}
  \label{cor:1}
  The maps \(\Psi_{\blank,r}\) define a  homomorphism of algebras
  \begin{equation}
    \label{eq:56}
    \Psi_r \colon \ucalH^\aff_r \mapto \End( \sfF^r).
  \end{equation}
\end{corollary}

We also define the functor
\begin{equation}
  \label{eq:58}
    \sfE \colon \catO \mapto \catO, \quad
    M  \mapto M \otimes (\C^\enne)^*.
\end{equation}
Note that \(\sfE\) and \(\sfF\) are biadjoint. We recall the following standard result, which is a direct consequence of the tensor identity (see \cite[Prop.~6.5]{MR938524}).
\begin{lemma}
  \label{lem:10}
  In the Grothendieck group \([\catO]\) we have
  \begin{equation}
    \label{eq:66}
    [\sfF M(\bolda)]  = \sum_{l=1}^\enne [M(\bolda + \epsilon_l)],\qquad
    [\sfE M(\bolda)]  = \sum_{l=1}^\enne [M(\bolda - \epsilon_l)].
  \end{equation}
\end{lemma}

\subsection{Combinatorics of weights}
\label{sec:comb-weights}

Recall that \(I=\{1,\dotsc,k\}\). The set \(I^\enne\) can be identified with a subset of the weights \(\weightsgl = \Z^n\) of \(\gl_\enne\). The weight from \(I^\enne\) are called \(k\)--\emph{bounded}. We define a map \(\phi\) from \(I^\enne\) to the set of weights of \(\mathfrak{sl}_k\) by 
\begin{equation}
\phi\colon \bolda=(a_1,\dotsc,a_\enne) \longmapsto \updelta_{a_1} + \dotsb + \updelta_{a_\enne}.\label{eq:95}
\end{equation}
 Notice that \(\phi\) is constant on the orbits of the action of the symmetric group \(W=\bbS_n\) on \(I^\enne\). In particular,  \(\phi^{-1}(\phi(\bolda))=W \bolda\) (this statement is a bit less obvious than it seems, since \(\updelta_{a_1}+\dotsb+\updelta_{a_n}\) is a weight for \(\mathfrak{sl}_k\) and not for \(\gl_k\)). We state the following result, whose easy proof follows directly from the definition:
\begin{lemma}
  \label{lem:9}
  Let \(\bolda \in I^\enne\) and  pick \(1 \leq l \leq \enne\). Set \(i=(\bolda,\epsilon_l)\). If  \(i<k\) then \(\phi(\bolda+ \epsilon_l) = \phi(\bolda) - \alpha_i\).
\end{lemma}

 For \(\boldd\in I^n\) dominant we let \((+i)\boldd\) denote the unique dominant weight with \(\phi((+i)\boldd)=\phi(\boldd)-\alpha_i\). Moreover, we let \((-i)\boldd\) denote the unique dominant weight with \(\phi((-i)\boldd)=\phi(\boldd)+\alpha_i\). Of course such weights do not always exist. If \((+i)\boldd\) is not defined then we just set \(\catO_{(+i)\boldd}=0\), and similarly for \((-i)\boldd\).
We denote by \(\catO_I\) the sum of all the blocks \(\catO_\boldd\) with \(\boldd \in I^\enne\) a dominant weight. 

Let us define \(\sfF_i \colon \catO_{\boldd} \mapto \catO_{(+i)\boldd}\) by \(\sfF_i = \pr_{(+i)\boldd} \circ \sfF\), where \(\pr_{(+i)\boldd}\colon \catO \mapto \catO_{(+i)\boldd}\) is the projection. Analogously, let us define \(\sfE_i \colon \catO_{\boldd} \mapto \catO_{(-i)\boldd}\) by \(\sfE_i = \pr_{(-i)\boldd} \circ \sfE\). We denote also by \(\sfF_i\) and \(\sfE_i\) the functors 
\begin{equation}
\sfF_i = \bigoplus_{\boldd \in I} \sfF_i\colon \catO_I \longrightarrow \catO_I, \qquad \sfE_i = \bigoplus_{\boldd \in I} \sfE_i \colon \catO_{I} \longrightarrow \catO_{I}.\label{eq:96}
\end{equation}

The following result appeared already at several places in the literature (see \cite[\textsection{}7.4]{MR2373155}, \cite[\textsection{}4.4]{MR2456464}). For convenience we give a complete proof in our setup.

\begin{lemma}
  \label{lem:7}
  The subfunctor \(\sfF_i\)
is the generalized \(i\)--eigenspace of \(x\) acting on \(\sfF\). 
That is, for all \(M \in \catO_I\) we have
\begin{equation}
  \label{eq:46}
  \sfF_i M = \{ v \in (M \otimes \C^\enne) \suchthat (\Omega-i)^N v = 0 \text{ for some } N \gg 0\}.
\end{equation}
\end{lemma}

\begin{proof}
The action of \(x\) on \(\sfF M\) is the action of the Casimir operator \(\Omega\) on \(M \otimes \C^\enne\). Since \(\catO\) has enough projective modules, and since each projective module has a Verma filtration, it is enough to consider the case \(M=M(\bolda)\).  We have
  \begin{equation}
    \label{eq:59}
    \Omega = \frac{1}{2} \big( \Delta(C) - C \otimes 1 - 1 \otimes C\big).
  \end{equation}
  The action of \(C\) on the Verma module \(M(\bolda)\) of highest weight \(\bolda - \rhogl\) is given by \(( \bolda - \rhogl, \bolda+\rhogl )\) (cf.\ \cite[Exercise~23.3.4]{MR499562}), hence the generalized eigenvalues of \(\Omega\) acting on \(M(\bolda) \otimes \C^\enne\) are
 \begin{equation}
   \label{eq:60}
   \frac{1}{2} \big(( \bolda + \epsilon_l-\rhogl,\bolda+ \epsilon_l+\rhogl) - ( \bolda-\rhogl, \bolda+\rhogl) - ( \epsilon_1,\epsilon_1+2\rhogl) \big) = ( \bolda, \epsilon_l )
 \end{equation}
 where \(l=1,\dotsc,\enne\). Now, given \(M(\bolda) \in \catO_{\boldd}\), by \Cref{lem:9} we have \(M(\bolda+\epsilon_l) \in \catO_{(+i)\boldd} \) if and only if \(( \bolda, \epsilon_l)=i\). The claim follows.
\end{proof}

\subsection{Categorical \texorpdfstring{$\fraksl_k$--}{sl(k)-}actions}
\label{sec:fraksl_k-categ}

We recall the definition of an \(\fraksl_k\)--categorification: 

\begin{definition}[{\cite[Definition~2.6]{2013arXiv1310.0349B}, cf.\ also \cite[Definition~5.29]{2008arXiv0812.5023R}}]
  \label{def:12}
  An \emph{\(\mathfrak{sl}_k\)--categorification} is a Schurian category \(\calA\) together with a pair of adjoint endofunctors \((\sfF,\sfE)\) and natural transformations \(x \in \End(\sfF)\), \(t \in \End(\sfF^2)\) such that:
  \begin{enumerate}[label=(SL\arabic*),leftmargin=*]
  \item \label{item:12} We have \(\sfF= \bigoplus_{i=1}^{k-1} \sfF_i\), where \(\sfF_i\) is the generalized \(i\)--eigenspace of \(x\).
  \item \label{item:13} For all \(d \geq 0\) the endomorphisms \(x_j=\sfF^{d-j} x \sfF^{j-1}\) and \(t_k = \sfF^{d-k-1} t \sfF^{k-1}\) of \(\sfF^d\) satisfy the relations of the degenerate affine Hecke algebra.
  \item \label{item:14} The functor \(\sfF\) is isomorphic to a right adjoint of \(\sfE\).
  \item \label{item:15} The endomorphisms \(f_i\) and \(e_i\) of \([\calA]\) induced by \(\sfF_i\) and \(\sfE_i\), respectively, turn \([\calA]\) into an integrable representation of \(\mathfrak{sl}_k\). Moreover, the classes of the indecomposable projective objects are weight vectors.
  \end{enumerate}
\end{definition}

If the Grothendieck group \([\calA]\) is isomorphic, as an \(\fraksl_k\)--representation, to \(V\), we say also that \(\calA\) is an \(\fraksl_k\)--\emph{categorification} of \(V\). If \(V= \bigoplus_{\nu \in \weightssl} V_\nu\) is the weight decomposition of \(V\), then by \cite{2008arXiv0812.5023R} \(\calA\) decomposes as \(\calA = \bigoplus_{\nu \in \weightssl} \calA_\nu\) where \(\calA_\nu = \{ M \in \calA \suchthat [M] \in V_\nu\}\). We will also say that \(\calA_\nu\) is the weight \(\nu\) subcategory.

\begin{prop}
  \label{prop:9}
  The data of the two exact functors \(\sfE, \sfF\) on \(\catO_{I}\) and of the natural transformations \(x \in \End(\sfF)\) and \(t \in \End(\sfF^2)\) define an \(\mathfrak{sl}_k\)--categorification of \(V^{\otimes \enne}\).
\end{prop}

\begin{proof}
  We need to check the conditions \crefrange{item:12}{item:15} above.
  Condition \ref{item:12} follows directly from \Cref{lem:7}. The action on \(\sfF^n\) of \(x_j\) and \(t_k\) induces an action of the degenerate affine Hecke algebra by \Cref{prop:7}, hence we have \ref{item:13}. The pair of functors \(\sfE\) and \(\sfF\) are biadjoint, hence \ref{item:14} is also true. From \Cref{lem:10} it follows that
 the endomorphisms \(f_i\) and \(e_i\) induced by \(\sfF_i\) and \(\sfE_i\), respectively, make \([\catO_{I}]\) into a representation of \(\mathfrak{sl}_k\) isomorphic to \(V^{\otimes \enne}\). Any projective module \(P(\bolda)\) has a filtration with Verma modules \(M(w\bolda)\) for \(w \in \bbS_\enne\). Since the isomorphism \([\catO_{I}] \mapto V^{\otimes \enne}\) sends all these Verma modules to standard basis vectors in the same weight space, it follows that \([P(\bolda)]\) is a weight vector, granting \ref{item:15}.
\end{proof}

We need the notion of an \(\mathfrak{sl}_k^{\oplus r}\)--categorification,  (cf.\ \cite{2008arXiv0812.5023R}, \cite{2013arXiv1303.1336L}):

\begin{definition}
  An \(\fraksl_k^{\oplus m}\)--categorification is a category \(\calA\)  which has \(m\) structures of an
  \(\mathfrak{sl}_k\)--categorification with functors
  \(\prescript{}{j}{\sfF}\), \(\prescript{}{j}{\sfE}\), for
  \(j=1,\dotsc,m\). These structures commute with each other, in the
  sense that we have natural isomorphisms
  \(\prescript{}{j}{\sfF} \prescript{}{h}{\sfF} \cong
  \prescript{}{h}{\sfF}\prescript{}{j}{\sfF}\) for all \(j,h
  =1,\dotsc,m\) which commute with the 2--morphisms \(x\)
  and \(t\) of \(\prescript{}{j}{\sfF}\).\label{def:17}
\end{definition}

The following result is straightforward.

\begin{prop}
  \label{prop:5}
  The outer tensor product \(\calA_1 \boxtimes \calA_2\) of two \(\fraksl_k\)--categorifications is an \(\fraksl_k^{\oplus 2}\)--categorification.
\end{prop}

\section{Categorification of simple representations}
\label{sec:categ-simple-repr}

Following \cite{2008arXiv0812.5023R} and \cite{2013arXiv1303.1336L} we define the categorification of a finite-dimensional irreducible \(\mathfrak{sl}_k\)--representation:

\begin{definition}
  \label{def:13}
  Let \(V(\lambda)\) be the finite-dimensional irreducible \(\mathfrak{sl}_k\)--representation of highest weight \(\lambda\). An \(\mathfrak{sl}_k\)--\emph{categorification} of \(V(\lambda)\) is an \(\mathfrak{sl}_k\)--categorification \(\calA\) such that its weight \(\lambda\) subcategory \(\calA_\lambda\) is equivalent to \(\Vect\) and the  \(\mathfrak{sl}_k\)--representation \([\calA]\) is isomorphic to \(V(\lambda)\).
\end{definition}

According to \cite{2013arXiv1303.1336L}, a categorification of \(V(\lambda)\) exists and is unique up to strongly equivariant equivalence. We present now a construction using the BGG category \(\catO\).

\subsection{Categorification of \texorpdfstring{$V$}{the vector representation}}
\label{sec:categorification-ck}

We define the categorification $\ucalC(V)$ of $V$ to be data of 
 the category $\catO^{(1)}_{I}$, the endofunctor $\sfF$ of $ \catO^{(1)}_{I}$ together with its right-adjoint functor $\sfE$, and the natural transformations \(x\in \End(\sfF)\) and \(t \in \End(\sfF^2)\).

 \begin{lemma}
   \label{lem:6}
   This defines an \(\mathfrak{sl}_k\)--categorification of \(V\).
 \end{lemma}

 \begin{proof}
   We already know by \Cref{prop:9} that \(\ucalC(V)\) is an \(\mathfrak{sl}_k\)--categorification. Notice that \(V=V(\varpi_1)\) and the  subcategory of highest weight \(\catO^{(1)}_{\varpi_1}\) is \(\catO^{(1)}_{1}\), which is equivalent to \(\Vect\). Hence we only need to observe that the \(\mathfrak{sl}_k\)--representation \([\catO^{(1)}_{I}]\) is isomorphic to \(V\), which is obvious.
 \end{proof}

\subsection{Categorification of \texorpdfstring{$V(\varpi_r)$}{fundamental representations}}
\label{sec:categ-vvarp}

For $r=1,\ldots,k-1$ we define the categorification $\ucalC(V(\varpi_r))$ of $V(\varpi_r)$  to be data of the category $\catO^{(r)}_{I}$,
the endofunctor $\sfF$ of $ \catO^{(r)}_{I}$  together with its right-adjoint $\sfE$ and the natural transformations $x \in \End(\sfF)$ and $t \in \End(\sfF^2)$.

\begin{lemma}
  \label{lem:8}
  This defines an \(\mathfrak{sl}_k\)--categorification of \(V(\varpi_r)\).
\end{lemma}

\begin{proof}
  Let \(\catO=\catO(\gl_r)\). Notice first that the endofunctors \(\sfF\) and \(\sfE\) of \(\catO\) restrict to \(\catO^{(r)}\) by definition of the parabolic category \(\catO\). Now, \(\ucalC(V(\varpi_r))\) satisfies \cref{item:12,item:13,item:14} automatically since \(\catO^{(r)}\) is a full subcategory of \(\catO\). Since it is actually a Serre subcategory, its Grothendieck group is naturally a subgroup of \([\catO]\), and the endomorphisms \(f_i\) and \(e_i\) induced on \([\catO^{(r)}_{I}]\) are just the restrictions of the endomorphisms \(f_i\) and \(e_i\) induced on \([\catO_{I}]\). Hence they satisfy the relations of \(U(\mathfrak{sl}_k)\) and turn \([\catO^{(r)}_{I}]\) into an \(\mathfrak{sl}_k\)--subrepresentation of \([\catO_{I}]\), and \ref{item:15} follows as well.\\
Since the simple object \(L(\bolda)\) of \(\catO_{I}\) belongs to \(\catO^{(r)}_{I}\) if and only if \(\bolda \in \weightsgl^+(\gl_r)\), that is if and only if \(\bolda\) is a strictly decreasing sequence, we have that the weight spaces of \([\catO^{(r)}_{I}]\) correspond to the weight spaces of \(V(\varpi_r)\), hence these two \(\mathfrak{sl}_k\)--representations have to be isomorphic.\\
Finally, note that the  subcategory of highest weight \(\catO^{(r)}_{\varpi_r}\) is just \(\catO^{(r)}_{(r,r-1,\dotsc,1)}\), which is equivalent to \(\Vect\). (Actually, all blocks of \(\catO^{(r)}\) are equivalent to \(\Vect\), or trivial.)
\end{proof}

\subsection{Categorification of \texorpdfstring{$V(\lambda)$}{arbitrary simple representations}}
\label{sec:categorification-v}

Given an arbitrary weight \(\lambda\) for \(\mathfrak{sl}_k\),
we define the categorification $\ucalC(V(\lambda))$ of $V(\lambda)$ to be data of 
the category $\calQ^{\lambda}_{I}$,
the endofunctor $\sfF$ of $ \calQ^{\lambda}_{I}$ together with its right-adjoint \(\sfE\), and the natural transformations $x \in \End(\sfF)$ and $t \in \End(\sfF^2)$.

If $\lambda$ is a fundamental weight $\varpi_r$ the definition coincides with the previous one.

\begin{prop}
  \label{prop:12}
 The data \(\ucalC(V(\lambda))\) defines an \(\mathfrak{sl}_k\)--categorification of \(V(\lambda)\).
\end{prop}
\begin{proof}
  Let \(\catO=\catO(\gl_\enne)\), where \(\lambda\) is a partition of \(n\). We first point out again that the endofunctors \(\sfF\) and \(\sfE\) of \(\catO\) restrict to \(\calQ^\lambda\) by definition. The properties \cref{item:12,item:13,item:14} are satisfied, since \(\calQ^\lambda\) is a full subcategory of \(\catO\). By definition, \(\calQ^\lambda\) is also a full subcategory of \(\catO^\lambda\). Since the latter is a Serre subcategory of \(\catO\), it follows as in the proof of \Cref{lem:8} that \([\catO^\lambda_{I}]\) is an \(\mathfrak{sl}_k\)--representation. The classes of the indecomposable projective modules \(P(\bolda)\), where \(\bolda\) is such that \(\tableau^\lambda(\bolda)\) is a column-strict tableau, form a basis of its Grothendieck group. The indecomposable prinjective objects  are the \(P(\bolda)\)'s, where \(\bolda\) is such that \(\tableau^\lambda(\bolda)\) is semi-standard, and they generate an additive subcategory which is stable under the action of \(\sfF\) and \(\sfE\). Moreover, their classes give a basis of \([\calQ^\lambda_{I}]\). It follows that \([\calQ^\lambda_{I}]\) is an \(\mathfrak{sl}_k\)--subrepresentation of \([\catO^\lambda_{I}]\), and hence \ref{item:15} holds.
Its highest weight corresponds to the semi-standard tableau \(T^{\text{high}}=\tableau^\lambda(\bolda^{\text{high}})\) with the smallest possible entries; in \(T^{\text{high}}\), the entry \(1\) appears once in each column, the entry \(2\) appears once in each column with at least two boxes, and so on. Hence it is easy to see that \(\phi(\bolda^{\text{high}})=\lambda\). Since the dimension of \([\calQ^\lambda_{I}]\) equals the number of semi-standard tableaux of shape \(\lambda\), hence coincides by \Cref{lem:22} with the dimension of \(V(\lambda)\), they must be isomorphic.
Finally, since \(\bolda^{\text{high}}\) is unique in its \(W\)--orbit such that the corresponding tableau of shape \(\lambda\) is column-strict, the summand of \(\calQ^\lambda_I\) corresponding to the highest weight \(\lambda\) is equivalent to \(\Vect\).
\end{proof}

\section{Categorification of tensor products}
\label{sec:categ-tens-prod}

We fix a sequence \(\boldlambda=(\lambda^{(1)},\dotsc,\lambda^{(m)})\) of integral dominant weights \(\lambda^{(l)} \in \weightssl^+\) and consider the (ordered) tensor product
\begin{equation}
V(\boldlambda)=V(\lambda^{(1)}) \otimes \dotsb \otimes V(\lambda^{(m)}).\label{eq:20}
\end{equation}
We want to construct a categorification \(\ucalC(V(\boldlambda))\) of the \emph{ordered} tensor product in the sense of \cite[Definition~3.2]{2013arXiv1303.1336L}:

\begin{definition}
  \label{def:15}
  A categorification of the \emph{ordered} tensor product \(V(\boldlambda)\) is the data of an \(\mathfrak{sl}_k\)--categorification \(\calA\) with endofunctors \(\sfF\), \(\sfE\) and natural transformations \(x \in \End(\sfF)\), \(t \in \End(\sfF^2)\) and \(\calA\) has the structure of a standardly stratified category with poset \(\Xi\).
These data must satisfy the following conditions:
  \begin{enumerate}[label=(TPC\arabic*),leftmargin=*]
  \item \label{item:16} The poset \(\Xi\) is the set of \(m\)--tuples \(\boldnu=(\nu_1,\dotsc,\nu_m)\), where \(\nu_l\) is a weight of \(V(\lambda^{(l)})\).  The preorder is given by the inverse dominance order:
    \begin{equation}
      \label{eq:61}
      \boldnu \preleq \boldnu' \quad \text{if }
      \begin{cases}
   \sum_{l=1}^m \nu_i = \sum_{l=1}^m \nu_i' & \text{and}\\
 \sum_{l=1}^h \nu_l \geq \nu_l' & \text{for all \(h\)}.
      \end{cases}
    \end{equation}
  \item \label{item:17} The associated graded category \(\gr \calA\) is an \(\mathfrak{sl}_k^{\oplus m}\)--categorification such that the subcategory of weight \(\boldnu\) is precisely the subquotient  \(\calA_{\preceq \boldnu}/\calA_{\prec \boldnu}\), and \(\calA_{\boldlambda} \cong \Vect\).
  \item \label{item:18} For each \(M \in \calC_{\boldnu}\) the objects \(\sfF_i \Delta_\boldnu(M)\) and \(\sfE_i \Delta_\boldnu(M)\) admit filtrations with successive quotients being \(\Delta(\prescript{}{j}{\sfF_i}M)\) and \(\Delta(\prescript{}{j}{\sfE_i}M)\), respectively, for \(j=1,\dotsc,m\).
\end{enumerate}
\end{definition}

Set \(\calQ^\boldlambda_I= \bigoplus_\boldd \calQ^\boldlambda_\boldd\), where \(\boldd\) runs over all dominant weights \(\boldd \in I^\enne\) for \(\gl_\enne\). Let \(\Xi_I = \sqcup_\boldd \Xi_\boldd\) and let \(p \colon \St^\boldlambda(I) \mapto \Xi_I\) be the disjoint union of the maps \(p_\boldd \colon \St^\boldlambda(\boldd) \mapto \Xi_\boldd\).

We define  $\ucalC(V(\boldlambda))$ to be  the data of the standardly stratified category $\calQ^{\boldlambda}_{I}$ with poset \(\Xi_I\), endofunctor $\sfF$ with right-adjoint \(\sfE\), and morphisms \(x \in \End(\sfF)\) and \(t \in \End(\sfF^2)\).

\begin{theorem}
  \label{thm:1}
  The data $\ucalC(V(\boldlambda))$ is a categorification of the tensor product \(V(\boldlambda)\) according to \Cref{def:15}.
\end{theorem}

\begin{proof}
  We need to check the three axioms \cref{item:16,item:17,item:18}.\\
We start with \ref{item:16}.  The simple objects of \(\calQ^\boldlambda_{I}\) are indexed by the set \(\St^\boldlambda(I)\). The poset \(\Xi_I\) can be identified with the set \(Z\) of \(k\)--bounded dominant weights for \(\frakl\), where \(\frakl=\gl_{n_1} \oplus \dotsb \oplus \gl_{n_m}\). Via \(m\) copies \(\phi_l\colon \weightsgl(\gl_{n_l}) \mapto \weightssl\) for \(l=1,\dotsc,m\) of the map \(\phi\), this can be further identified with the set of \(m\)--tuples \(\boldnu=(\nu_1,\dotsc,\nu_m)\) where \(\nu_l\) is a weight of \(V(\boldlambda^{(l)})\). The order on \(\Xi_I\) is given by restricting the dominance order on each \(W\)--orbit of \(Z \subset \weightsgl^+(\frakl)\). Since this is generated by simple reflections \(s_1,\dotsc,s_{\enne-1} \in W\) it is enough to consider the case  of some \(\bolda \in \weightsgl(\frakl)\) with \(s_h \bolda \leq \bolda\), i.e.\ \(a_h \geq a_{h+1}\). If both the \(h\)--th and the \((h+1)\)--th entries of \(\bolda\) belong to the same component of the multipartition \(\boldlambda\), then by definition \(\phi_l(\bolda)=\phi_l(s_h \bolda)\) for all \(l\). Otherwise, there is an index \(l\) such that \(\phi_l(s_h \bolda)= \phi_l(\bolda) - \updelta_{a_h}+ \updelta_{a_{h+1}}\) and \(\phi_{l+1}(s_h \bolda) = \pi_{l+1}(\bolda) - \updelta_{a_{h+1}}+ \updelta_{a_h}\), while \(\phi_{l'}(s_h \bolda) = \phi_{l'} (\bolda)\) for all \(l' \neq l,l+1\). Since \(a_{h} \geq a_{h+1}\), we have \(\phi_l(s_h \bolda) \geq  \phi_l(\bolda)\), proving the claim.\\
To verify \ref{item:17} note that by \Cref{prop:6} the associated graded category is 
    \begin{equation}
\gr \calQ^\boldlambda_{I} \cong \textstyle\bigoplus_{\xi \in \Xi_I} \calC^\boldlambda_\xi, \label{eq:97}
\end{equation}
which is the outer tensor product of the categories \(\calC^{\lambda^{(1)}}_{I} \boxtimes \dotsb \boxtimes \calC^{\lambda^{(m)}}_{I}\) and hence carries the structure of an \(\mathfrak{sl}_k^{\oplus m}\)--categorification of \(V(\boldlambda)\) by \Cref{prop:5}.\\
Finally, we check \ref{item:18}. By the tensor identity we have
    \begin{equation}
      \label{eq:44}
      \C^\enne \otimes \Delta(M) \cong \Delta(\C^\enne \otimes M),
    \end{equation}
    where on the left we have a tensor product of \(\gl_\enne\)--representations and on the right a tensor product of \(\frakp_\frakl\)--representations.
   As a \(\frakp_\frakl\)--representation, \(\C^\enne\) is filtered  by
    \begin{equation}
      \label{eq:49}
      \{0\} = \C^0 \subseteq \C^{\enne_1} \subseteq \C^{\enne_1+\enne_2} \subseteq \dotsb \subseteq \C^\enne.
    \end{equation}
    This induces a filtration on \(\Delta(\C^\enne \otimes M)\), and hence on \(\C^\enne \otimes \Delta(M)\). By projecting onto \(\calQ^\boldlambda_I\), this gives a filtration on \(\sfF(\Delta(M))\)  with successive subquotients being \(\Delta(\prescript{}{j}{\sfF}(M))\) for \(j=1,\dotsc,m\), where \(\prescript{}{j}{\sfF}\) denotes the functor \cref{eq:57} on the \(j\)--th factor \(\calC^{\lambda^{(j)}}\) of the outer tensor product \(\calC^\boldlambda\). By projecting onto the right blocks, we get the wanted filtration for \(\sfF_i\). Analogously we get the required filtration for \(\sfE_i\).
\end{proof}

Note that the theorem implies Corollary \ref{Cor1} from the introduction using the uniqueness result of \cite{2013arXiv1303.1336L} and the fact from \cite{2013arXiv1303.1336L} that the Webster algebras give rise to a tensor product categorification as well. 

\section{Graded category \texorpdfstring{\(\catO\)}{O}}
\label{sec:graded-lift-soergel}

In order to lift the categorifications we constructed so far to \(U_q(\fraksl_k)\)--categorifications, we need to recall some basic facts about the graded version of the category \(\catO\).

\subsection{Soergel's functor}
\label{sec:soergels-functor}

The key-tool to construct a graded lift of the category \(\catO\) is Soergel's functor \(\bbV\) from \cite{MR1029692}. Fix \(\frakl=\gl_{\enne_1} \oplus \dotsb \oplus \gl_{\enne_r}\), and let \(\enne=\enne_1 + \dotsb + \enne_r\). Let \(W_\frakl \subseteq \bbS_\enne\) be the Weyl group of \(\frakl\). For all \(\boldd \in \weightsgl^+(\frakl)\) we let \(R_\boldd = \C[x_1(\boldd),\dotsc,x_\enne(\boldd)]\), where we add the symbol \(\boldd\)  to the variable \(x_h\) in order to stress that \(x_h(\boldd)\) belongs to \(R_\boldd\). Let \(\bbS_\boldd \subseteq W_\frakl\) denote the stabilizer of \(\boldd\) (notice that this depends not only on \(\boldd\), but also on \(\frakl\)) and let
\begin{equation}
  \label{eq:65}
  C_{\frakl,\boldd} = \big(R_\boldd/(R_\boldd^+)^{W_\frakl} R_\boldd\big)^{\bbS_\boldd}
\end{equation}
denote the invariants for \(\bbS_\boldd\) inside the algebra of the coinvariants. As usual, \((R_\boldd^+)^{W_\frakl}R_\boldd\) is the ideal generated by polynomials with zero constant term which are symmetric in the variables \(x_{\enne_h+1}(\boldd),\dotsb,x_{\enne_{h+1}}(\boldd)\) for all \(h=1,\dotsb,r\).  Let \(w_{\frakl,0}\) denote the longest element of \(W_\frakl\). Soergel's Endomorphismensatz \cite{MR1029692} provides a canonical identification
\begin{equation}
\End_\frakl(P(w_{\frakl,0}\boldd)) = C_{\frakl,\boldd}.\label{eq:74}
\end{equation}
 We denote by
\begin{equation}
  \label{eq:68}
  \bbV_{\frakl,\boldd} = \Hom_\frakl(P(w_{\frakl,0} \boldd),\blank) \colon \catO(\frakl)_\boldd \mapto \lmod{C_{\frakl,\boldd}}
\end{equation}
Soergel's functor, and set 
\begin{equation}
  \label{eq:67}
  C_\frakl = \bigoplus_{\boldd \in \weightsgl^+(\frakl)} C_{\frakl,\boldd}
\qquad \text{and} \qquad  \bbV_\frakl = \bigoplus_{\boldd \in \weightsgl^+(\frakl)} \bbV_{\frakl,\boldd} \colon \catO(\frakl) \mapto \lmod{C_\frakl}.
\end{equation}
By Soergel's Struktursatz \cite{MR1029692},  \(\bbV_\frakl\) is fully faithful on projective objects.

\subsection{Graded category \texorpdfstring{$\catO$}{O}}
\label{sec:graded-category-cato}
Let us denote by \(P_{\frakl,\boldd}\) a minimal projective generator of \(\catO(\frakl)_\boldd\), and let \(A_{\frakl,\boldd} = \End_{C_{\frakl,\boldd}}(\bbV_\frakl P_{\frakl,\boldd})\). Then since \(\End_\frakl(P_{\frakl,\boldd}) \cong \End_{C_{\frakl,\boldd}} (\bbV_\frakl P_{\frakl,\boldd})\) we have \(\catO(\frakl)_\boldd \cong \rmod{A_{\frakl,\boldd}}\). Moreover  \(A_\frakl := \bigoplus_{\boldd \in \weightsgl^+(\frakl)} A_{\frakl,\boldd}\) is a locally unital algebra and 
\begin{equation}
\catO(\frakl) \cong \rmod{A_\frakl}.\label{eq:87}
\end{equation}
The algebra \(A_\frakl\) admits a natural grading (see \cite[Section~2]{MR2005290} and \cite{MR1322847}). As a consequence, one defines the \emph{graded category \(\catO(\frakl)\)} to be
\begin{equation}
  \label{eq:85}
  \catOZ(\frakl) = \rgmod{A_\frakl}.
\end{equation}
The algebra \(A_\frakl\) is a positively graded locally unital algebra, and the grading is the unique Koszul grading \cite{MR1322847}. The primitive idempotents projecting onto the indecomposable modules are homogeneous of degree \(0\). It follows that Serre subcategories and Serre quotients of \(\rgmod{A_\frakl}\) inherit a natural grading. In this way we define graded versions \(\catOZ^\boldlambda\) and \(\calQZ^\boldlambda\) of \(\catO^\boldlambda\) and \(\calQ^\boldlambda\), respectively.

\subsection{Graded lifts of functors}
\label{sec:grad-lifts-funct}

Let \(\sfG\colon \catO(\frakg) \mapto \catO(\frakg')\) be a right-exact functor, which is compatible with direct sums. Then under the equivalence of categories \cref{eq:87}, (cf.\ \cite[Lemma~3.4]{MR2017061} and \cite[2.2]{MR0249491}),  we have that \(\sfG\) is isomorphic to the tensor product functor \(\blank \otimes_{A_\frakg} \Hom_{A_{\frakg'}}(A_{\frakg'},\sfG A_{\frakg})\)

Suppose now that \(\sfG\) sends projectives to projectives, and suppose moreover that we have an isomorphism of functors \(\bbV_{\frakg'} \sfG \cong \sfG' \bbV_{\frakg'}\), where \(\sfG'\) is a functor \(\lmod{C_{\frakg}} \mapto \lmod{C_{\frakg'}}\). Then by Soergel's theorem we have 
\begin{equation*}
\Hom_{A_{\frakg'}} (A_{\frakg'}, \sfG A_{\frakg}) \cong \Hom_{C_{\frakg'}}(\bbV_{\frakg'} A_{\frakg'}, \bbV_{\frakg'} \sfG A_{\frakg}) \cong
\Hom_{C_{\frakg'}}(\bbV_{\frakg'} A_{\frakg'}, \sfG' \bbV_{\frakg} A_{\frakg}). \label{eq:103}
\end{equation*}
It follows that if  \(\sfG'\) admits a graded lift \(\tilde \sfG' \colon \gmod{C_\frakg} \mapto \gmod{C_{\frakg'}}\) then  \(\sfG\) as well.

\subsection{Graded translation functors}
\label{sec:grad-transl-funct}

Let now \(\frakg=\gl_\enne\). Recall that for all \(i \in \Z\) we have adjoint functors \(\sfF_i,\sfE_i\colon \catO(\frakg) \mapto \catO(\frakg)\).  Their restrictions to single blocks \(\sfF_i\colon \catO(\frakg)_{\boldd} \mapto \catO(\frakg)_{(+i)\boldd}\) and \(\sfE_i \colon \catO(\frakg)_{\boldd} \mapto \catO(\frakg)_{(-i)\boldd}\) satisfy the following properties:

\begin{lemma}
\label{lem:grading}
The functors \(\sfF_i\colon \catO(\frakg)_{\boldd} \mapto \catO(\frakg)_{(+i)\boldd}\) and \(\sfE_i \colon \catO(\frakg)_{\boldd} \mapto \catO(\frakg)_{(-i)\boldd}\) are indecomposable exact functors. Hence a graded lift is unique up to isomorphism and overall grading shift. 
\end{lemma}

\begin{proof}
The functors are exact by definition. The uniqueness follows then from the indecomposability by standard arguments, see e.g. \cite[Proposition 3.11]{MR2120117}. For the indecomposability it is enough to show that the Verma module with maximal possible weight in the block is sent to an indecomposable projective object $P$, \cite[Theorem 3.1]{MR2120117}. By Lemma \ref{lem:10} $P$ has a Verma filtration such that every Verma module appears at most once. Hence the socle of $P$ must be simple by \cite[Theorem 8.1]{MR2017061} and so $P$ is indecomposable.
\end{proof}

Fix \(\boldd \in \weightsgl^+(\frakg)\), and let \(\boldd' = (+i)\boldd\). Notice that the sequences \(\boldd\) and \(\boldd'\) differ at most in one place, and that \((-i)\boldd' = \boldd\). Let
\begin{equation}
\prescript{(+i)}{}{C}_{\frakg,\boldd} = (R_\boldd/(R_\boldd^+)^{W_\frakg} R_\boldd)^{\bbS_\boldd \cap \bbS_{\boldd'}}.\label{eq:78}
\end{equation}
If we define in an analogous way \(\prescript{(-i)}{}{C}_{\frakg,\boldd'}\), then we have an obvious isomorphism \(\prescript{(+i)}{}{C}_{\frakg,\boldd} \cong \prescript{(-i)}{}{C}_{\frakg,\boldd'}\) given by renaming the variables \(x_i(\boldd) \mapsto x_i(\boldd')\). We have natural inclusions \(C_{\frakg,\boldd},C_{\frakg,\boldd'} \into \prescript{(+i)}{}{C}_{\frakg,\boldd}\) which turn \(\prescript{(+i)}{}{C}_{\frakg,\boldd}\) into a \((C_{\frakg,\boldd'},C_{\frakg,\boldd})\)--bimodule and \(\prescript{(-i)}{}{C}_{\frakg,\boldd'}\) into a \((C_{\frakg,\boldd},C_{\frakg,\boldd'})\)--bimodule.
By summing up we get \(C_\frakg\)--bimodules \(\prescript{(+i)}{}{C}_{\frakg} = \bigoplus_{\boldd \in \weightsgl^+(\frakg)} \prescript{(+i)}{}{C}_{\frakg,\boldd}\) and \(\prescript{(-i)}{}{C}_{\frakg} = \bigoplus_{\boldd \in \weightsgl^+(\frakg)} \prescript{(-i)}{}{C}_{\frakg,\boldd}\).

\begin{prop}
  \label{prop:15}
  We have isomorphisms of functors
  \begin{equation}
    \label{eq:84}
    \bbV_{\frakg}\sfF_i \cong  \prescript{(+i)}{}{C}_{\frakg} \otimes_{C_{\frakg}} \bbV_{\frakg} \qquad \text{and} \qquad  \bbV_{\frakg}\sfE_i \cong  \prescript{(-i)}{}{C}_{\frakg} \otimes_{C_{\frakg}} \bbV_{\frakg}
  \end{equation}
\end{prop}

\begin{proof}
  Let us prove the first isomorphism, the second one being analogous. Of course it is sufficient to check that for each \(\boldd \in \weightsgl^+(\frakg)\) we have
  \begin{equation}
    \label{eq:83}
    \bbV_{\frakg,\boldd'}\sfF_i|_{\catO(\frakg)_\boldd} \cong C_{\frakg,\boldd'} \otimes_{C_{\frakg,\boldd'}} \prescript{(+i)}{}{C}_{\frakg,\boldd} \otimes_{C_{\frakg,\boldd}} \bbV_{\frakg,\boldd}.
  \end{equation}
  Let \(\sfT_\boldd^\boldb\colon \catO(\frakg)_\boldd \mapto \catO(\frakg)_\boldb\) be the usual translation functor. Fix \(\boldd \in \Z^\enne_\geq\), and let \(\boldd' = (+i)\boldd\). Let also \(\boldb \in \bbZ^\enne_\geq\) be an integral dominant weight with stabilizer \(\bbS_\boldd \cap \bbS_{\boldd'}\). It follows then from the classification of projective functors that \(\sfF_i|_{\catO(\frakg)_{\boldd}} \cong \sfT_\boldb^{\boldd'} \circ \sfT_\boldd^\boldb\). But \(\sfT_\boldb^{\boldd'}\) is a translation functor onto a wall, while \(\sfT_\boldd^\boldb\) is a translation functor out of a wall. Then \cref{eq:83} follows from \cite[Theorem~10]{MR1029692}.
\end{proof}

We fix graded shifts \(\sfFZ_i\) and \(\sfEZ_i\) of \(\sfF_i\) and \(\sfE_i\) by setting
\begin{align}
  \sfFZ_i|_{\catO(\frakg)_\boldd} &= \blank \otimes_{A_{\frakg,\boldd}} \Hom_{C_{\frakg,\boldd}} (\bbV_\frakg A_{\frakg,\boldd}, \prescript{(+i)}{}{C}_{\frakg,\boldd} \otimes_{C_{\frakg,\boldd}} \bbV_{\frakg,\boldd} A_{\frakg,\boldd} \langle -c_{\boldd,i} \rangle)\\
  \sfEZ_i|_{\catO(\frakg)_\boldd} &= \blank \otimes_{A_{\frakg,\boldd}} \Hom_{C_{\frakg,\boldd}} (\bbV_\frakg A_{\frakg,\boldd}, \prescript{(-i)}{}{C}_{\frakg,\boldd} \otimes_{C_{\frakg,\boldd}} \bbV_{\frakg,\boldd} A_{\frakg,\boldd} \langle -c_{\boldd,i+1} \rangle)
\end{align}
where
\begin{equation}
  \label{eq:104}
  c_{\boldd,i} = \#\{h \suchthat d_h = i\}.
\end{equation}

\begin{prop}
  \label{prop:19}
  We have a graded adjunction
  \begin{equation}
    \label{eq:105}
    \sfFZ_i |_{\catO(\frakg)_\boldd} \adjunction \sfEZ_i|_{\catO(\frakg)_{(+i)\boldd}} \langle c_{\boldd,i}- c_{\boldd,i+1} + 1 \rangle.
  \end{equation}
\end{prop}

\begin{proof}
  With the notation from the proof of \Cref{prop:15}, one can fix graded lifts of the translation functors so that we have \(\sfF_i|_{\catO(\frakg)_\boldd} \cong \sfTZ_\boldb^{\boldd'} \circ \sfTZ_\boldd^\boldb\) and \(\sfE_i|_{\catO(\frakg)_{(+i)\boldd}} \cong \sfTZ^\boldd_\boldb \circ \sfTZ^\boldb_{\boldd'}\) (see \cite[\textsection{}4.4]{miophd2}). We have graded adjunctions (see \cite[Lemma~4.4.1]{miophd2}, which is a consequence of the classification theorem of projective functors)
  \begin{equation}
    \label{eq:107}
    \sfTZ_\boldd^\boldb \adjunction \sfTZ_\boldb^\boldd \langle c_{\boldd,i} \rangle \qquad \text{and} \qquad \sfTZ_\boldb^{\boldd'} \adjunction \sfTZ_{\boldd'}^\boldb \langle - c_{\boldd',i+1}\rangle.
  \end{equation}
  Now \(\boldd' = (+i)\boldd\) and \(c_{\boldd',i+1}= c_{\boldd,i+1}+1\). Hence \(\sfTZ_\boldb^{\boldd'} \circ \sfTZ_\boldd^\boldb \adjunction \sfTZ_\boldb^\boldd \circ \sfTZ_{\boldd'}^\boldb \langle c_{\boldd,i} - c_{\boldd',i+1} - 1\rangle\), which is our claim.
\end{proof}

\subsection{Graded standardization functor}
\label{sec:graded-functor-delta}

As before let \(\frakl=\gl_{\enne_1}\oplus \dotsb\oplus \gl_{\enne_r}\) be a standard Levi subalgebra in \(\frakg\). Let \(\boldd \in \weightsgl^+(\frakg)\) and \(\Xi=\Xi_\boldd\) and \(p=p_\boldd\)  as in \cref{sec:stand-funct}. 

\begin{lemma}
  \label{lem:21}
  Let \(\xi \in \Xi\). The functor \(\Delta_\xi \colon \catO(\frakl)_{\xi} \mapto \catO(\frakg)_{\preleq \xi}\) is gradable. A graded lift is determined uniquely by fixing the degree shift on \(\Delta_\xi(M(\tilde \xi))\), where \(M(\tilde \xi)\) is the dominant Verma module of \(\catO(\frakl)_\xi\).
\end{lemma}

\begin{proof}
   By \Cref{prop:21}, \(\catO(\frakl)_\xi\) is equivalent to the Serre quotient \(\catO(\frakg)_{\preleq \xi}/\catO(\frakg)_{\prel \xi}\).
 Under this equivalence \(\Delta_\xi\), being left adjoint to \(\pi_\xi\), becomes the inclusion functor of the Serre quotient category in \(\catO(\frakg)_{\preleq \xi}\). In particular, if \(B\) is the endomorphism algebra of a minimal projective generator of \(\catO(\frakg)_{\preleq \xi}\) and \(e\) is the idempotent projecting onto the indecomposable projective modules which are not in \(\catO(\frakg)_{\prel \xi}\), then \(\Delta_\xi\) corresponds to \(\blank \otimes_{eBe} eB \colon \rmod{eBe} \mapto \rmod{B}\) (see \cref{sec:serre-subc-serre-1}).
It is then clear, see the proof of Lemma \ref{lem:grading}, that the indecomposable bimodule \(eB\) admits a graded lift, unique up to a shift. The shift is uniquely determined by an object which is not killed by \(\Delta_\xi\), for example \(M(\tilde \xi)\).
\end{proof}

We fix the graded lift \(\DeltaZ\) so that \(\DeltaZ_\xi(M(\tilde \xi))=M(\tilde \xi)\). Notice that it follows also that the functor \(\Delta_\xi\) restricted to \(\calC^\boldlambda_\xi\) is gradable.
It follows also that each block of \(\calQ^\boldlambda\) is graded standardly stratified, with standardization functor \(\DeltaZ\).

\section{Graded categorification}
\label{sec:grad-categ-1}

We are going now to construct graded lifts of our categorifications to categorifications of  \(U_q(\fraksl_k)\)--representations. Notice that we already know by abstract reasons (see \cite[Corollary~6.3]{2013arXiv1303.1336L}) that such graded lift exist. Our goal here is to realize them explicitly (up to some extent) in the Lie theoretical setting.

In order to modify \Cref{def:12} for obtaining  \(U_q(\fraksl_k)\)--categorifications, we need to replace the action of the degenerate affine Hecke algebra with a graded action of a quiver Hecke algebra. In the ungraded setting, it is equivalent to require an action of the degenerate affine Hecke algebra or of the quiver Hecke algebra: this follows from the remarkable isomorphism between cyclotomic quotients of the two algebras proved in \cite{MR2551762}.  Now, the degenerate affine Hecke algebra does not come with any natural grading, while the quiver Hecke algebra does. Hence it is natural, for an \(U_q(\fraksl_k)\)--categorification, to require a graded action of the quiver Hecke algebra on the functors.

We do not want to enter into details, which are explained in \cite[Section~2]{2013arXiv1310.0349B}, neither we will define the quiver Hecke algebra. All the reader needs to know is that such an algebra exists and admits a natural grading, and that in \Cref{def:12} one could equivalently replace \ref{item:13} by the requirement of an (ungraded) action of the quiver Hecke algebra.

We recall that a graded category is a category \(\calCg\) with an autoequivalence \(\langle 1\rangle \colon \calCg \mapto \calCg\), which we will also denote by \(q\). A graded category \(\calCg\) is called \emph{acyclic} if \(q^\ell L \neq L\) for all \(\ell \neq 0\) and \(L \in \calCg\) irreducible.

\begin{definition}[{\cite[Definition~5.5]{2013arXiv1310.0349B}}]
  \label{def:1}
  An \emph{\(U_q(\mathfrak{sl}_k)\)--categorification} is an
  acyclic graded Schurian category \(\calCg\) together with
  graded endofunctors \(\sfFg_i\), \(\sfEg_i\), \(\sfKg_i\)
  and \(\sfKg_i^{-1}\) for \(i=1,\dotsc,k-1\), an adjunction
  making \(q\sfEg_i \sfKg_i\) into a right adjoint to
  \(\sfFg_i\), and homogeneous natural transformations
  \(\xig \in \End(\sfFg_i)_2\), \(\taug \in \Hom(\sfFg_j
  \sfFg_i, \sfFg_i \sfFg_j)_{- i \cdot j}\) for each \(i,j =
  1,\dotsc,k-1\) such that:
  \begin{enumerate}[label=(GSL\arabic*),leftmargin=*]
  \item \label{item:19} There is a decomposition \(\calCg= \bigoplus_{\nu \in P} \calCg_\nu\) such that \(\sfKg_i|_{\calCg_\nu} \cong q^{(\nu, \alpha_i)}\) and \(\sfKg_i^{-1}|_{\calCg_\nu} \cong q^{-(\nu,\alpha_i)}\).
  \item \label{item:20} The natural transformations \(\xig\) and \(\taug\) define an action of the quiver Hecke algebra.
  \item \label{item:21} Each functor \(q\sfFg_i \sfKg_i^{-1}\) is isomorphic to a right adjoint of \(\sfEg_i\).
  \item \label{item:22} The endomorphisms \(F_i\) and \(E_i\) of \([\calCg]\) induced by \(\sfFg_i\) and \(\sfEg_i\), respectively, make \([\calCg]\) into an integrable representation of \(U_q(\mathfrak{sl}_k)\).
  \end{enumerate}
\end{definition}

We recall the following useful result:

\begin{lemma}[{\cite[Lemma~5.7]{2013arXiv1310.0349B}}]
  \label{lem:11}
  Let \( \calC\) be an \(\fraksl_k\)--categorification with functors \(\sfF_i\), \(\sfE_i\).
For each \(\nu \in P\) let \( \calC_\nu\) be the full subcategory of \( \calC\) consisting of all objects \( M\) such that \([ M]\) lies in the \(\nu\)--weight space of the \(\fraksl_k\)--module \([ \calC]\). Suppose that we are given the following additional data:
  \begin{enumerate}[(1)]
  \item \label{item:23} a graded lift \(\calCg_\nu\) of each \( \calC_\nu\);
  \item \label{item:26} graded functors \(\sfKg_i\) and \(\sfKg_i^{-1}\) satisfying \ref{item:19};
  \item \label{item:24} graded lifts \(\sfFg_i\) and \(\sfEg_i\) of the functors \( \sfF_i\) and \( \sfE_i\) together with an adjunction making \(q\sfEg_i \sfKg_i\) into a right adjoint to \(\sfFg_i\);
  \item \label{item:25} a graded lift of the action of the quiver Hecke algebra.
  \end{enumerate}
Then \(\calCg\) is an \(U_q(\fraksl_k)\)--categorification.
\end{lemma}

By \cref{sec:graded-category-cato}, the category \(\catOZ(\gl_\enne)_I\) is a graded lift of \(\catO(\gl_\enne)_I\).
We define on \(\catOZ(\gl_\enne)\) functors \(\sfKZ_i\) and \(\sfKZ_i^{-1}\) by
\begin{equation}
  \label{eq:108}
  \sfKZ_i|_{\catO(\frakg)_\bolda}  = \langle c_{\bolda,i} - c_{\bolda,i+1}\rangle \qquad \text{and} \qquad \sfKZ_i^{-1}|_{\catO(\frakg)_\bolda} = \langle c_{\bolda,i+1} - c_{\bolda, i} \rangle,
\end{equation}
where \(c_{\bolda,i}\) was defined in \cref{eq:104}. Of course \(\sfKZ^{-1}_i \circ \sfKZ_i \cong \sfKZ_i \circ \sfKZ^{-1}_i \cong \id\), and obviously \(\sfKZ_i\) and \(\sfKZ_i^{-1}\) satisfy \ref{item:19}. By \cref{sec:grad-transl-funct}, moreover, we have graded lifts \(\sfFZ\) and \(\sfEZ\) of the endofunctors \(\sfF\) and \(\sfE\) of \(\catO\), which by \Cref{prop:19} satisfy the graded adjunction of \ref{item:24} above. If we prove that \ref{item:25} is also satisfied, then by \Cref{lem:11} we have an \(U_q(\fraksl_k)\)--categorification.

Condition~\ref{item:25} amounts to say that the action of the quiver Hecke algebra on the functors \(\sfFZ_i\) and \(\sfEZ_i\) is homogeneous. Unfortunately, as far as the authors know, there is no direct proof of this fact yet. Nevertheless, it is possible to conclude that \ref{item:25} holds by an indirect argument. Indeed, by the uniqueness result \cite[Theorem~6.1]{2013arXiv1303.1336L}) the categorification \(\catO(\gl_\enne)_I\) with its structure is strongly equivariantly equivalent to Webster's diagrammatic categorification by Corollary~\ref{Cor1}. The latter admits an explicit graded lift such that the corresponding graded algebra is Koszul (\cite[Proposition~8.11]{2013arXiv1309.3796W}). By the uniqueness of the Koszul grading, it follows that the equivalence lifts to an equivalence between \(\catOZ(\gl_\enne)_I\) and Webster's diagrammatic category. Actually, one can also deduce that the graded lifts \(\sfFZ_i\) and \(\sfFZ_i\) correspond to Webster's graded diagrammatic functors (see \cite[Corollary~8.12]{2013arXiv1309.3796W}). Hence the action of the quiver Hecke algebra of the categorification \(\catO(\gl_\enne)_I\) does admit a graded lift. Hence we obtain the following result:

\begin{prop}
  \label{prop:13}
  For all \(\enne \geq 0\) the category \(\catOZ(\gl_\enne)_{I}\) together with the endofunctors \(\sfFZ_i\), \(\sfEZ_i\), \(\sfKZ_i\) and \(\sfKZ^{-1}_i\) and the graded action of the quiver Hecke algebra is an \(U_q(\fraksl_k)\)--categorification.
\end{prop}

Let now \(\boldlambda\) be a sequence of integral dominant weights for \(\fraksl_k\). As the category \(\calQ^\boldlambda_I\) is a full subcategory of \(\catO(\gl_\enne)_I\), the graded lift \(\calQZ^\boldlambda_I\) is naturally a subcategory of \(\catOZ(\gl_\enne)_I\). The endofunctors \(\sfFZ_i\), \(\sfEZ_i\), \(\sfKZ_i\) and \(\sfKZ_i^{-1}\) of \(\catOZ(\gl_\enne)_I\) restrict to graded endofunctors of \(\calQZ^\lambda_I\), and the graded action of the quiver Hecke algebra on the former gives a graded action on the latter. As a consequence, the following graded version of \Cref{prop:12} follows immediately:

\begin{prop}
  \label{prop:20}
  Let \(\lambda \in \weightssl^+\). Then the data of the graded category \(\calQZ^\lambda_I\), of the endofunctors \(\sfFZ_i\), \(\sfEZ_i\), \(\sfKZ_i\) and \(\sfKZ_i^{-1}\) and of the graded action of the quiver Hecke algebra defines an \(U_q(\fraksl_k)\)--categorification of \(V_q(\lambda)\).
\end{prop}

We obtain the following graded version of \Cref{thm:1}:

\begin{theorem}
  \label{thm:4}
  Let \(\boldlambda\) be a sequence of integral dominant weights for \(\fraksl_k\). The data of the graded category \(\calQZ^\boldlambda_I\), of the endofunctors \(\sfFZ_i\), \(\sfEZ_i\), \(\sfKZ_i\) and \(\sfKZ_i^{-1}\) and of the graded action of the quiver Hecke algebra defines an \(U_q(\fraksl_k)\)--categorification of  \(V_q(\boldlambda)\), which is a graded lift of the categorification \(\ucalC(V(\boldlambda))\).
\end{theorem}

Now Corollary~\ref{Cor2} follows directly from Theorem~\ref{thm:4}, since the graded category $C(n)$  is by definition, see \cite{miophd2}, \cite{Sartori}, a subquotient category of \(\calQZ^\boldlambda_I\), with \(V(\lambda^{(i)})=V(\varpi_1)\) for $1\leq i\leq n$. 

\begin{remark}
  \label{rem:3}
  It is possible to define an \(U_q(\fraksl_k)\)--categorification of a tensor product, generalizing \cite[Definition~5.8]{2013arXiv1310.0349B}, essentially by replacing \cref{item:17,item:18} in \Cref{def:15} with the following graded versions: \begin{enumerate}[label=(GTPC\arabic*),leftmargin=*]
    \setcounter{enumi}{1}
  \item \label{item:g17} The associated category \(\gr \calCg\) is an \(U_q(\mathfrak{sl}_k^{\oplus r})\)--categorification such that the subcategory of weight \(\boldnu\) is precisely the subquotient \(\calCg_{\preceq \boldnu}/\calCg_{\prec \boldnu}\). Moreover \(\calCg_{\boldlambda} \cong \gmod{\C}\).
  \item \label{item:g18} For each \(M \in \calCg_{\boldnu}\) the objects \(\sfFg_i \Deltag_\boldnu(M)\) and \(\sfEg_i \Deltag_\boldnu(M)\) admit a filtration with successive quotients being graded shifts of \(\Deltag(\prescript{}{j}{\sfFg_i}M)\) and \(\Deltag(\prescript{}{j}{\sfEg_i}M)\), respectively, for \(j=1,\dotsc,r\).
  \end{enumerate}
 It follows by the uniqueness result \cite[Corollary~6.3]{2013arXiv1303.1336L} (see also \cite[Theorem~5.10]{2013arXiv1310.0349B}) that an \(U_q(\fraksl_k)\)--categorification which is a graded lift of a tensor product categorification is equivalent to Webster's diagrammatic categorification. Since the latter satisfies \cref{item:g17,item:g18}, the former also does. We remark that it is easy to prove \ref{item:g17} for \(\calQZ^\boldlambda_I\) exactly as in the non-graded setting. We however do not know an easy argument to check \ref{item:g18} explicitly for \(\calQZ^\boldlambda_I\).
\end{remark}

\providecommand{\bysame}{\leavevmode\hbox to3em{\hrulefill}\thinspace}
\providecommand{\MR}{\relax\ifhmode\unskip\space\fi MR }
\providecommand{\MRhref}[2]{%
  \href{http://www.ams.org/mathscinet-getitem?mr=#1}{#2}
}
\providecommand{\href}[2]{#2}

\end{document}